\documentclass[a4paper,10pt,withmarginpar]{amsart}
\usepackage{amsthm,amssymb}
\usepackage{stmaryrd}
\usepackage[T1]{fontenc}
\usepackage[utf8]{inputenc}
\usepackage{enumitem}
\usepackage{mathrsfs}
\usepackage{dsfont}
\usepackage{mathtools}
\usepackage{scalerel,stackengine}
\usepackage{nicefrac}
\usepackage{enumitem}
\usepackage[margin=3.5cm]{geometry}
\usepackage{tikz}
\usetikzlibrary{positioning}

\usepackage{natbib}

\usepackage[colorlinks,citecolor=blue,urlcolor=blue]{hyperref}

\numberwithin{equation}{section}

\makeatletter
\DeclareSymbolFont{lettersA}{U}{txmia}{m}{it}
\SetSymbolFont{lettersA}{bold}{U}{txmia}{bx}{it}
\DeclareFontSubstitution{U}{txmia}{m}{it}

\DeclareMathSymbol{\m@thbbch@rA}{\mathord}{lettersA}{129}
\DeclareMathSymbol{\m@thbbch@rB}{\mathord}{lettersA}{130}
\DeclareMathSymbol{\m@thbbch@rC}{\mathord}{lettersA}{131}
\DeclareMathSymbol{\m@thbbch@rD}{\mathord}{lettersA}{132}
\DeclareMathSymbol{\m@thbbch@rE}{\mathord}{lettersA}{133}
\DeclareMathSymbol{\m@thbbch@rF}{\mathord}{lettersA}{134}
\DeclareMathSymbol{\m@thbbch@rG}{\mathord}{lettersA}{135}
\DeclareMathSymbol{\m@thbbch@rH}{\mathord}{lettersA}{136}
\DeclareMathSymbol{\m@thbbch@rI}{\mathord}{lettersA}{137}
\DeclareMathSymbol{\m@thbbch@rJ}{\mathord}{lettersA}{138}
\DeclareMathSymbol{\m@thbbch@rK}{\mathord}{lettersA}{139}
\DeclareMathSymbol{\m@thbbch@rL}{\mathord}{lettersA}{140}
\DeclareMathSymbol{\m@thbbch@rM}{\mathord}{lettersA}{141}
\DeclareMathSymbol{\m@thbbch@rN}{\mathord}{lettersA}{142}
\DeclareMathSymbol{\m@thbbch@rO}{\mathord}{lettersA}{143}
\DeclareMathSymbol{\m@thbbch@rP}{\mathord}{lettersA}{144}
\DeclareMathSymbol{\m@thbbch@rQ}{\mathord}{lettersA}{145}
\DeclareMathSymbol{\m@thbbch@rR}{\mathord}{lettersA}{146}
\DeclareMathSymbol{\m@thbbch@rS}{\mathord}{lettersA}{147}
\DeclareMathSymbol{\m@thbbch@rT}{\mathord}{lettersA}{148}
\DeclareMathSymbol{\m@thbbch@rU}{\mathord}{lettersA}{149}
\DeclareMathSymbol{\m@thbbch@rV}{\mathord}{lettersA}{150}
\DeclareMathSymbol{\m@thbbch@rW}{\mathord}{lettersA}{151}
\DeclareMathSymbol{\m@thbbch@rX}{\mathord}{lettersA}{152}
\DeclareMathSymbol{\m@thbbch@rY}{\mathord}{lettersA}{153}
\DeclareMathSymbol{\m@thbbch@rZ}{\mathord}{lettersA}{154}

\long\def\DoLongFutureLet #1#2#3#4{%
   \def\@FutureLetDecide{#1#2\@FutureLetToken
      \def\@FutureLetNext{#3}\else
      \def\@FutureLetNext{#4}\fi\@FutureLetNext}
   \futurelet\@FutureLetToken\@FutureLetDecide}
\def\DoFutureLet #1#2#3#4{\DoLongFutureLet{#1}{#2}{#3}{#4}}
\def\@EachCharacter{\DoFutureLet{\ifx}{\@EndEachCharacter}%
   {\@EachCharacterDone}{\@PickUpTheCharacter}}
\def\m@keCharacter#1{\csname\F@ntPrefix#1\endcsname}
\def\@PickUpTheCharacter#1{\m@keCharacter{#1}\@EachCharacter}
\def\@EachCharacterDone \@EndEachCharacter{}

\DeclareRobustCommand*{\varmathbb}[1]{\gdef\F@ntPrefix{m@thbbch@r}%
  \@EachCharacter #1\@EndEachCharacter}
\makeatother

\newtheorem{theorem}{Theorem}[section]
\newtheorem{lemma}[theorem]{Lemma}
\newtheorem{proposition}[theorem]{Proposition}
\newtheorem{corollary}[theorem]{Corollary}

\newtheoremstyle{mytheoremstyle} 
    {1em plus .2em minus .1em}                    
    {1em plus .2em minus .1em}                    
    {\rmfamily}                   
    {}                           
    {\bfseries}                   
    {.}                          
    {.5em}                       
    {}  

\theoremstyle{mytheoremstyle}

\newtheorem{definition}[theorem]{Definition}
\newtheorem{example}[theorem]{Example}

\newcommand{\zero}{\mathbf{0}} 
\newcommand{\one}{\mathbf{1}} 

\newcommand\Rarrow{\Rightarrow}
\newcommand\Larrow{\Leftarrow}
\newcommand\rarrow{\rightarrow}

\newcommand\iffdef{\;\mathrel{\mathord{:}\mathord{\longleftrightarrow}}\;}
\newcommand{\power}{\mathcal{P}} 
\newcommand\defeq{\coloneqq} 
\newcommand\dftt{\mathtt{df}\,}

\DeclareMathOperator{\Id}{Id} 
\DeclareMathOperator{\closed}{C}
\DeclareMathOperator{\clopen}{CO}
\let\topo=\tau
\newcommand{\topos}{\topo_s}
\DeclareMathOperator{\Ul}{Ul}
\let\Ult=\Ul
\DeclareMathOperator{\Fi}{Fi}
\DeclareMathOperator{\Cong}{Con} 
\let\Fil=\Fi

\newcommand\Es{\ensuremath{\operatorname{{\mathsf{Es}}}}}
\newcommand\Cm{\ensuremath{\operatorname{{\mathsf{Cm}}}}}
\newcommand\Em{\ensuremath{\operatorname{{\mathsf{Em}}}}}

\renewcommand{\zero}{0}
\renewcommand{\one}{1}
\DeclareMathOperator{\upop}{\uparrow} 
\DeclareMathOperator{\downop}{\downarrow} 
\newcommand\Nat{\varmathbb{N}}
\newcommand{\Even}{\varmathbb{E}} 

\newcommand{\CA}{\mathsf{CA}} 
\newcommand{\MMA}{\mathsf{MMA}} 
\newcommand\PSB{\mathsf{PSB}}
\newcommand\PsC{\mathsf{PsC}}

\newcommand\SIA{\mathsf{SIA}}
\newcommand{\StwoIA}{\mathsf{S^2IA}}

\newcommand{\ult}{u} 
\newcommand{\ultV}{v}
\newcommand{\ultW}{w}
\newcommand{\fil}{\mathscr{F}} 
\newcommand{\scrH}{\mathscr{H}}
\newcommand{\scrG}{\mathscr{G}}
\newcommand{\filH}{\scrH}
\newcommand{\filG}{\scrG}

\newcommand{\ide}{\mathscr{I}} 

\newcommand{\barT}{\overline{T}}

\renewcommand{\mid}{:}

\newcommand\condi{\rightarrowtriangle}
\newcommand\Tcondi{\condi_{T}}

\newcommand\barTcondi{\condi_{\barT}}

\newcommand{\TprimePX}{T'_{\power(X)}}
\newcommand{\barTprimePX}{\barT'_{\power(X)}}
\newcommand\TAcondi{\condi_{T_A}}

\newcommand\D[2]{D^{\condi}_{#1}(#2)}

\newcommand{\Uf}{\ensuremath{\operatorname{{\mathsf{Uf}}}}}
\newcommand{\Co}{\ensuremath{\operatorname{{\mathsf{Co}}}}}
\newcommand{\tand}{\text{ and }}
\newcommand{\qtiff}{\quad\text{iff}\quad}

\newcommand\klam[1]{\left\langle#1\right\rangle}

\newcommand*{\QED}{\null\nobreak\hfill\ensuremath{\mathord{\dashv}}}
\newcommand{\trarrow}{\condi}

\newcommand\mfr{\mathfrak}
\newcommand\frA{\mfr{A}}
\newcommand\frB{\mfr{B}}
\newcommand\frM{\mfr{M}}
\newcommand\frN{\mfr{N}}
\newcommand\frX{\mfr{X}}

\newcommand\calF{\mathcal{F}}

\newcommand\mbf{\mathbf}
\newcommand\bfF{\mbf{F}}

\title{Conditional algebras}

\author[]{Sergio Celani, Rafa\l\ Gruszczy\'nski and Paula Mench\'on}

\date{}

\address{Sergio Celani, \textsc{Orcid}: 0000-0003-2542-4128\\ Conicet and
University of the Center of the Buenos Aires Province, Tandil (Unicen)\\
Argentina}
\email{sergiocelani@gmail.com}

\address{Rafa\l\ Gruszczy\'nski, \textsc{Orcid:} 0000-0002-3379-0577\\ Paula Mench\'on, \textsc{Orcid:} 0000-0002-9395-107X\\
Department of Logic\\
Nicolaus Copernicus University in Toru\'n\\
Poland}

\email{gruszka@umk.pl, paula.menchon@v.umk.com}

\begin{document}

\begin{abstract}
Drawing on the classic paper by \cite{Chellas-BCL}, we propose a general algebraic framework for studying a~binary operation of \emph{conditional} that models universal features of the ``if \ldots, then \ldots'' connective as strictly related to the unary modal necessity operator. To this end, we introduce a variety of \emph{conditional algebras}, and we develop its duality and canonical extensions theory.

\smallskip
  
  \noindent MSC: Primary 06E25, Secondary 03G05

\smallskip
  
  \noindent Keywords: conditionals, Boolean algebras, Boolean algebras with operators, conditional algebras, modal algebras, pseudo-subordination algebras, binary operators
\end{abstract}

\maketitle

\section{Introduction}

In the paper ``Basic conditional logic'' \citeyearpar{Chellas-BCL}, Brian Chellas put forward a family of propositional systems whose goal was to capture the properties of a binary connective $\rightarrowtail$, the so-called \emph{conditional}.\footnote{Chellas used the `$\Rarrow$' symbol. As we reserve this one for our meta-implication connective, we have chosen `$\rightarrowtail$' instead.} The presentation and the analysis were based on a~simple and appealing idea: if-then sentences are closely involved with relative necessity. Since a~sentence with the form ``if $p$, then $q$'' conveys dependency of the content of its consequent $q$ on the content of the antecedent $p$, it may be interpreted in the following way: $q$ must be true, provided that $p$ is true, or---to put it differently---$q$ obtains in all those states in which $p$~holds true. Thus, on the most elementary level, the conditional conceived as an operator is strictly related to the well-known necessity operator.

The three basic axioms of Chellas's---that result from the above-mentioned idea---were
\begin{gather}
    p\rightarrowtail\top\tag{CN}\,,\\
    (p\rightarrowtail(q\wedge r))\rarrow ((p\rightarrowtail q)\wedge(p\rightarrowtail r))\,,\tag{CM}\\
    ((p\rightarrowtail q)\wedge(p\rightarrowtail r))\rarrow (p\rightarrowtail(q\wedge r))\,,\tag{CC}
\end{gather}
and the system based on them (and the classical logic) was named \emph{normal conditional logic}. A special kind of frame semantics (with a function $f\colon W\times\power(W)\to\power(W)$ in lieu of the routine ternary relation in $W^3$) was developed, and the standard metatheorems connecting it with the logic were proven.\footnote{Further analysis of the Chellas approach to conditionals was carried out by \cite{Segerberg-NOCL}.}

It is easy to see that each axiom has its counterpart in which $\rightarrowtail$ is replaced with the family of relative necessity operators $[p]$, one for each variable $p$. Thinking of $p\rightarrowtail q$ as $[p]q$ ($q$ is necessary relative to $p$), instead of the three axioms, we have infinitely many of them as instances of the following axioms schemata:
\begin{gather*}
    [\alpha]\top\,,\\
    [\alpha](q\wedge r)\rarrow ([\alpha]q\wedge[\alpha]r)\,,\\
    ([\alpha]q\wedge[\alpha]r)\rarrow [\alpha](q\wedge r)\,.
\end{gather*}
where $\alpha$ marks the place in which we can put any propositional variable from the language. These explain the choice of the three axioms for the conditional logic.

The algebraic semantics for Chellas's logic (and related systems) was provided by \cite{Nute-TICL} in the form of Boolean algebras with a~binary operator $\condi$ interpreting $\rightarrowtail$.\footnote{Originally, Nute used `$\ast$' instead of our `$\condi$'.} He defined classes of algebras satisfying various conditions put upon $\condi$, of which the variety of normal regular algebras satisfying conditions:
\begin{gather}
    a\condi 1=1\,,\tag{C1}\label{C1}\\
    (a\condi b)\wedge(a\condi c)=a\condi(b\wedge c)\,,\tag{C2}\label{C2}
\end{gather}
corresponds directly to Chellas's normal conditional logic. Among others, Nute proved the completeness theorem for this logic and normal regular algebras.\footnote{A similar algebraic semantics for the basic conditional logic of Chellas, but based on Heyting algebras, was presented in \citep{Weiss-BICL}.}

In this paper, we aim to focus our attention on and delve into the study of Boolean algebras expanded with the $\condi$ operator that meets \eqref{C1} and \eqref{C2}, and which in light of Nute's results are algebraic models of Chellas's normal conditional logic. The consequence of \eqref{C2} is that, in the second coordinate, $\condi$ is isotone, which opens the possibility of utilizing the algebraic tools developed in \citep{Celani-TDFBAWANNMO,Menchon-PhD,Celani-et-al-MDS} for monotonic operators. To be able to do this, we must ensure full monotonicity, which can easily be obtained by adopting the third condition
\begin{equation}\tag{C3}\label{C3}
(a\vee b)\condi c\leq(a\condi c)\wedge(b\condi c)\,,
\end{equation}
that, in the event, is equivalent to $\condi$ being antitone in the first coordinate. Our goals justify the choice, but the condition is also the counterpart of the following axiom considered by Chellas
\begin{equation}\tag{CM$'$}\label{CM-prime}
    ((p\vee q)\rightarrowtail r)\rarrow ((p\rightarrowtail r)\wedge (q\rightarrowtail r))\,,
\end{equation}
being one of the axioms considered for conditionals.\footnote{The algebraic counterpart of \eqref{CM-prime} is missing from Nute's study of conditional logic, yet his completeness theorem extends to the normal conditional logic with \eqref{CM-prime} and the variety of conditional algebras. Let us also observe, after \citep[Section 3.5]{Egre-et-al-TLOC}, that the propositional version of this axiom (referred to as Simplification of Disjunctive
Antecedents) is ``a bone of contention between theorists''. Although considered intuitively valid by many logicians, it was shown to entail monotonicity of conditionals, which is considered as one of the three paradigmatic invalidities of conditional logic (next to transitivity and contraposition). However, the monotonicity can be obtained by means of a certain rule of inference called LLE in \citep[Section 3.3]{Egre-et-al-TLOC}. The rule has been abandoned by some scholars in preference for the axiom (see e.g., \citealp{Nute-TICL}; \citealp{Fine-CWPW}; \citealp{Ciardelli-et-al-TWITTOC}; \citealp{Santorio-AATINS}), but some other contemporary treatments of the logic of conditionals do not mention SDA at all (see \citealp{Unterhuber-et-al-CACICSS}; \citealp{Weiss-FOCL}).  In our environment, there is no risk of obtaining the undesired monotonicity, and our hereby presented work may be considered as a~contribution to the algebraic analysis of the very basic conditional logic validating SDA.}
Thus, we adopt the following
\begin{definition}\label{df:conditional-algebra} $\frA\defeq\langle A,\condi\rangle$ is a \emph{conditional} algebra if $A$ is a Boolean algebra and $\condi$ is a~binary operation on $A$ that satisfies \eqref{C1}, \eqref{C2}, and \eqref{C3}. The $\condi$ operation will be called the \emph{conditional}. The class of conditional algebras is a variety, which we will denote by means of $\CA$. \QED
\end{definition}

In the sequel, we begin with the development of the theory of canonical extensions of conditional algebras. In particular, we prove that $\CA$ is closed under such extensions. Then, we introduce the notion of a \emph{multimodal antitone algebra}, which is a Boolean algebra $A$ with a family of unary necessity operators $\Box_a$ (indexed by the elements of a subalgebra of $A$) corresponding to $\condi$ in the same way as Chellas's $[p]$'s correspond to $\rightarrowtail$. We prove that conditional algebras and multi-modal antitone algebras are term equivalent, and we investigate the relation between their canonical extensions. Having finished this, we develop topological and categorical dualities for conditional algebras, and we characterize their subalgebras and congruences. 

From a purely algebraic point of view, $\CA$ is a generalization of some well-known varieties, such as:
\begin{enumerate}
    \item  pseudo-contact algebras \citep{Duntsch-et-al-RBTODSAPA}, and their equivalent counterparts subordination algebras \citep{Bezhanishvili-G-et-al-IERGSADVD},
    \item pseudo-subordination algebras \citep{Celani-et-al-AVOACRTSO},
    \item strict-implication algebras and its subvariety of symmetric strict-implication algebras \citep{Bezhanishvili-at-al-ASICFCHS}.
\end{enumerate}
In the last section of the paper, we characterize these as subvarieties of $\CA$, using the tools developed in earlier sections.

\subsection{Notational conventions}

If $\leq$ is a partial order on a set $X$, and  $Y\subseteq X$, then
\[
\upop Y\defeq\{x\in X:\exists y\in Y(y\leq x)\}
\]
in an \emph{upward closure} of $Y$, and 
\[
\downop Y\defeq\{x\in X:\exists y\in Y(x\leq y)\}
\]
is its \emph{downward closure}. If $Y=\{y\}$, then we will write $\upop y$ and $\downop y$ instead of $\upop\{y\}$
and $\downop\{y\}$, respectively. We call $Y$ an \emph{upset} (resp.
\emph{downset}) if $Y=\upop Y$ (resp. $Y=\downop Y$). 

The set-theoretical complement of
a subset $Y\subseteq X$ will be denoted by $Y^{c}$. Similarly, if $R$ is an $n$-ary relation, then $R^c$ is its complement, and we write $R(x_1,\ldots,x_n)$ (resp. $R^c(x_1,\ldots,x_n)$) instead of $\klam{x_1,\ldots,x_n}\in R$ (resp. $\klam{x_1,\ldots,x_n}\notin R$). If $f\colon X \to Y$ is a function and $U \subseteq X$, then $f[U]=\{f(x):x\in U\}$ is the \emph{direct image} of $U$ through $f$. If $V \subseteq Y$, then $f^{-1}[V]=\{x:f(x)\in V\}$ is the \emph{inverse image} of $V$ through $f$.

For a topological space $\langle X,\topo\rangle$,  $\closed(\tau)$ and $\clopen(\tau)$ are, respectively, the families of its closed and clopen subsets. 

$\langle A,\vee,\wedge,\neg,0,1\rangle$ is a Boolean algebra with the operations of, respectively, join, meet, and complement, and two constants, bottom and top. We usually identify the Boolean algebra with its domain. 

For a Boolean algebra $A$, $\Fi(A)$ and $\Ul(A)$ are sets of all its, respectively, filters and ultrafilters. We assume that the domain of $A$ is an element of $\Fi(A)$, the only \emph{improper} filter of $A$, and that ultrafilters are maximal sets (w.r.t. set theoretical inclusion) in the set of all \emph{proper} filters of $A$. We will use calligraphic letters `$\fil$', `$\filH$', `$\filG$' to denote filters, and letters `$\ult$', `$\ultV$', `$\ultW$' to range over ultrafilters. 

The family of open sets of the Stone space $\Ul(A)$ of a Boolean algebra $A$ will be denoted by~`$\topos$'.
$\varphi\colon A\condi\power(\Ul(A))$ is the standard Stone mapping 
\[
\varphi(a)\defeq\{\ult\in\Ul(A):a\in\ult\}.
\]
As is well known, the family $\{ \varphi(a):a\in A\} $
is a field of subsets of $\Ul(A)$, and therefore is a Boolean subalgebra
of the algebra $\mathcal{P}(\Ul(A))$. If $\fil$ is a filter of $A$, by means of $\varphi(\fil)$---abusing the notation slightly---we will denote the set $\{\ult\in\Ul(A): \fil\subseteq\ult\}$ of all ultrafilters extending $\fil$. Recall that $\varphi(\fil)=\bigcap\{ \varphi(a):a\in\fil\}$. Given $Y\subseteq\Ul(A)$, $\fil_Y$ is the filter $\{a\in A\mid Y\subseteq\varphi(a)\}$. If $Y$ is closed, then $Y=\varphi(\fil_Y)$. 

\section{Conditional algebras and their canonical extensions}

\begin{lemma}\label{lem:orden}
    In every conditional algebra $\frA\defeq\langle A,\condi\rangle$\/\textup{:}
    \begin{gather}
    \text{If $a\leq b$, then $x\condi a\leq x\condi b$.}\label{eq:cond-order-preserving}\\
        \text{If $a\leq b$, then $b\condi x\leq a\condi x$.}\label{eq:cond-order-reversing}\\
        (a\condi b)\wedge(x\condi y)\leq (a\wedge x)\condi(b\wedge y)\,.\label{eq:cond-conjunction-twice}  \end{gather}
\end{lemma}
\begin{proof}
    \eqref{eq:cond-order-preserving} Let $a=a\wedge b$. Then by \eqref{C2}
    \[
        x\condi a=x\condi (a\wedge b)=(x\condi a)\wedge (x\condi b)\,.
    \]

    \eqref{eq:cond-order-reversing} This time, let $b=a\vee b$. Then applying \eqref{C3} we get
    \[
    b\condi x=(a\vee b)\condi x\leq(a\condi x)\wedge (b\condi x)\,. 
    \]

    \eqref{eq:cond-conjunction-twice} Since $a\wedge x\leq a,x$ we apply the previous point and \eqref{C2}
    \begin{align*}
(a\condi b)\wedge(x\condi y)&{}\leq [(a\wedge x)\condi b]\wedge [(a\wedge x)\condi y]\\
&{}=(a\wedge x)\condi (b\wedge y)\,.\qedhere
    \end{align*}
\end{proof}

Since---unlike, e.g., in the case of pseudo-subordination algebras---conditional operators are not in one-to-one correspondence with binary normal modal operators, it is necessary to employ a more expressive framework than ternary relations between ultrafilters for their representation. To build ultrafilter frames for conditional algebras we are going to need a ternary hybrid relation (i.e., involving both points and sets of points) associated with the conditional operator. Given $\langle A,\condi\rangle\in\CA$, and $X,Y\subseteq A$ we define the set
\begin{align*}
\D{X}{Y}\defeq{}&\{b\in A:(\exists a\in Y)\,a\condi b\in X\}\\
={}&\{\pi_2(\klam{a,b})\mid a\in Y\tand a\condi b\in X\}\,,
\end{align*}
where $\pi_2$ is the standard projection on the second coordinate. Intuitively, $\D{X}{Y}$ is the set of all consequents of those conditionals in $X$ whose antecedents are in $Y$.

\noindent

\begin{lemma}\label{lem:D-filter}
    Let $\frA\defeq\klam{A,\condi}\in\CA$. 
    \begin{enumerate}\itemsep3pt
        \item If $Y_1\subseteq Y_2\subseteq A$, then $\D{X}{Y_1}\subseteq\D{X}{Y_2}$, for any $X\subseteq A$.
        \item If $Y$ is upward closed, then $\D{Y}{X}$ is too, for any $X\subseteq A$.
        \item $D_{\filH}^{\condi}(\fil)\in\Fi(A)$, for all $\fil,\filH\in\Fi(A)$. 
    \end{enumerate}
\end{lemma}
\begin{proof}
Ad 1. If $b\in\D{X}{Y_1}$, then by definition there exists $a\in Y_1$ such that $a\condi b\in X$. But $a\in Y_2$ by the assumption and we get that $b\in\D{X}{Y_1}$.

\smallskip

Ad 2. Let $a\in\D{Y}{X}$ and $a\leq b$. By definition, in $X$ there is a $c$ such that $c\condi a\in Y$. From \eqref{eq:cond-order-preserving} we get that $c\condi a\leq c\condi b$ and since $Y$ is upward closed, $c\condi b\in Y$. In consequence $b\in\D{Y}{X}$, as required.

\smallskip 

Ad 3. From the previous proposition, we get that $D_{\filH}^{\condi}(\fil)$ is upward closed. Further, $1\in \fil$ and $1\condi 1=1\in \filH$.
Thus, $1\in D_{\filH}^{\condi}(\fil)$. 

To show that the set is closed under infima, take its elements $a$ and $b$. So there are $c_{1},c_{2}\in \fil$
such that $c_{1}\condi a,c_{2}\condi b\in \filH$. By the fact that $\filH$ is
a filter and by \eqref{eq:cond-conjunction-twice} we get that: $c_{1}\wedge c_{2}\condi a\wedge b\in \filH$.
As $c_{1}\wedge c_{2}\in \fil$, $a\wedge b\in D_{\filH}^{\condi}(\fil)$.

This concludes the proof.
\end{proof}

\begin{definition}\label{df:ultrafilter-hybrid-frame}
For a conditional algebra $\frA\defeq\langle A,\condi\rangle$ the \emph{ultrafilter frame} of $\frA$ is the structure $\Uf(\frA)\defeq\langle\Ul(A),T_A\rangle$,
with $T_A\subseteq\Ul(A)\times\mathcal{P}(\Ul(A))\times\Ul(A)$
such that
\begin{equation}\tag{$\mathrm{df}\,T_A$}\label{df:T_A}
T_A(\ult,Z,\ultV)\quad\text{iff}\quad (\exists\fil\in\Fi(A))(Z=\varphi(\fil)\wedge D_{\ult}^{\condi}(\fil)\subseteq\ultV)\,.
\end{equation}
Thus, $T_A$ holds between points and closed subsets of the Stone space of~$A$. \QED 
\end{definition}

Given a set $Y\subseteq \Ul(A)$ let
\[
T_A(\ult,Y)\defeq\{\ultV\in \Ul(A):T_A(\ult,Y,\ultV)\}\,.
\]
Since
\begin{equation}\label{eq:remark-T}
\begin{split}
    T_A(\ult,\varphi(\fil))&{}=\{\ultV\in \Ul(A):D_{\ult}^{\condi}(\fil)\subseteq\ultV\}\\
    &{}=\varphi(D_{\ult}^{\condi}(\fil))\,,
    \end{split}
\end{equation}
we get that
\begin{proposition}
$T_A(\ult,\varphi(\fil))$ is always a closed subset of the Stone space of~$A$.
\end{proposition}

The definition of the relation $T_A$ is strictly related to and motivated by the techniques and results from \citep{Celani-TDFBAWANNMO,Menchon-PhD,Celani-et-al-MDS}. The papers deal with monotone operators that are modeled by hybrid relations\footnote{These are called \emph{multirelations} in the aforementioned papers.} on its dual space, i.e., relations between ultrafilters and sets of ultrafilters. In particular, we get that, given a~Boolean algebra with a binary operator $f$ that is antitone in the first coordinate and isotone in the second, the  dual relation $R_f$ of $f$ is a set of triples in $\Ul(A)\times \power(\Ul(A))\times \power(\Ul(A))$ such that 
\begin{align*}
R_f(\ult,Z,Y)\quad\text{iff}\quad(\exists\fil\in\Fi(A))(\exists\ide\in\Id(A))&\,(Z=\varphi(\fil)\ \text{and}\ \\ 
Y=\varphi(\neg \ide)&\ \text{and}\ f^{-1}[\ult]\cap (\fil\times\ide)=\emptyset)
\end{align*}
where:
\begin{enumerate}\itemsep+2pt
    \item $\Id(A)$ is the set of ideals of $A$,
    \item $\neg \ide=\{\neg a:a\in\ide\}$, and
    \item $f^{-1}[\ult]=\{\klam{a,b}\in A^2:f(a,b)\in \ult\}$.
\end{enumerate}
Applying this to the particular case of the conditional operator, we first observe that 
\[
(\condi)^{-1}[\ult]\cap (\fil\times\ide)=\emptyset\quad\text{iff}\quad D^\condi_\ult(\fil)\cap \ide=\emptyset\,. 
\]
    Using Lemma~\ref{lem:D-filter} according to which $D^\condi_\ult(\fil)$ is a filter, we obtain that 
\[
 R_\condi(\ult,\varphi(\fil),\varphi(\neg\ide))\quad\text{iff}\quad (\exists\ultV\in\Ul(A))\,(D^\condi_\ult(\fil)\subseteq \ultV\ \text{and}\ \ultV\in \varphi(\neg\ide))\,. 
\]
In consequence, if $R_\condi(\ult,\varphi(\fil),\varphi(\neg\ide))$, then there exists an ultrafilter $\ultV$ extending $\neg\ide$ such that $R_\condi(\ult,\varphi(\fil),\{\ultV\})$.
So, replacing the singleton of $\ultV$ with $\ultV$ itself, we can simplify $R_\condi$ to the equivalent $T_A\subseteq \Ul(A)\times\power(\Ul(A))\times \Ul(A)$.

\subsection{Representation of conditional algebras}

\begin{definition}
Let $X$ be a set, and let $\calF_X$ be any family of subsets of $X$.  $\frX\defeq\langle X,T\rangle$ is a \emph{ternary hybrid frame} (abbr.\ \emph{t-frame}) if $T\subseteq X\times\calF_X\times X$. As earlier, let 
\[
T(x,Z)\defeq\{y\in X: T(x,Z,y)\}.
\]
\QED
\end{definition}

\begin{proposition}\label{prop:power-set-algebra}
If $\frX\defeq\langle X,T\rangle$ is a t-frame and  $\calF_X$ is a family of subsets of $X$ then the structure $\langle \calF_X,\Tcondi \rangle$ such that 
\begin{equation}\tag{$\dftt{\Tcondi}$}\label{df:Tcondi}
U\Tcondi V\defeq\{x\in X:(\forall Z\in\power(U)\cap\calF_X)\,T(x,Z)\subseteq V\} 
\end{equation}
is a conditional algebra.
\end{proposition}
\begin{proof}
\eqref{C1} Immediate.

\smallskip

\eqref{C2} Let $x\in (U\Tcondi V)\cap (U\Tcondi  W)$, $Z\in \mathcal{P}(U)\cap\calF_X$ and $y\in X$ be such that $T(x,Z,y)$. By assumption we get that $y\in V$ and $y\in W$, so $y\in V\cap W$. The converse is analogous. 

\eqref{C3} Let $x\in (U\cup V)\Tcondi W$, $Z\in \mathcal{P}(U)\cap\calF_X$ and $y\in X$ be such that $T(x,Z,y)$. Then, $Z\subseteq U\cup V$ and, by assumption, $y\in W$. The other inclusion is analogous.
\end{proof}
\begin{definition}
    Given a t-frame $\frX\defeq\langle X,T\rangle$, its \emph{full complex conditional algebra} is the structure $\Cm(\frX)\defeq\langle\power(X),\Tcondi\rangle$.\QED
\end{definition}

\begin{definition}
    Let $\frA\defeq\langle A,\condi\rangle\in\CA$. The full complex algebra of the ultrafilter frame $\Uf(\frA)$ is  \[
\Em(\frA)\defeq\Cm(\Uf(\frA))=\langle\mathcal{P}(\Ul(A)),\condi_{T_A}\rangle\,.
    \]
\end{definition}

\begin{lemma}\label{lem:existence-for-representation} Let $\frA\defeq\klam{A,\condi}\in\CA$. Let $\filH\in\Fi(A)$
and let $a,b\in A$. Then $a\condi b\in \filH$ if and only if for every filter
$\fil$ such that $a\in \fil$, $b\in D_{\filH}^{\condi}(\fil)$.

Consequently, $a\condi b\notin \filH$ if and only if there exist a filter $\fil$ and an ultrafilter
$\ult\in\Ul(A)$ such that $a\in \fil$, $D_{\filH}^{\condi}(\fil)\subseteq\ult$
and $b\notin\ult$.
\end{lemma}
\begin{proof}
($\Rarrow$) If $a\condi b\in \filH$ and $a\in\fil$, then $b\in\D{\filH}{\fil}$ by definition.

\smallskip 

($\Larrow$) Suppose that for any filter $\fil$ such that $a\in \fil$, $b\in D_{\filH}^{\condi}(\fil)$. In particular, $b\in\D{\filH}{\upop a}$, so there is $x\geq a$ such that $x\condi b\in\filH$. Since by \eqref{eq:cond-order-reversing} it is the case that $x\condi b\leq a\condi b$, we have that $a\condi b\in\filH$. 
\end{proof}

\begin{theorem}[Representation Theorem]\label{th:RT-for-CCA}
    For every $\frA\defeq\langle A,\condi\rangle\in\CA$, the Stone mapping $\varphi\colon A\to\power(\Ul(A))$ is an embedding of the algebra into $\Em(\frA)$.
\end{theorem}
\begin{proof}
    ($\Rarrow$) Suppose $\ult\in\varphi(a)\condi_{T_A}\varphi(b)$, i.e., $(\forall Z\subseteq\varphi(a))\,T_A(\ult,Z)\subseteq\varphi(b)$. Let $\fil$ be a filter such that $a\in\fil$. Thus $\varphi(\fil)\subseteq\varphi(a)$ and so $T_A(\ult,\varphi(\fil))\subseteq\varphi(b)$. By \eqref{eq:remark-T} we obtain that $\varphi(\D{\ult}{\fil})\subseteq\varphi(b)$, which entails that $b\in\D{\ult}{\fil}$. But this, according to Lemma~\ref{lem:existence-for-representation} means that $a\condi b\in\ult$. 

    \smallskip

    ($\Larrow$) In this case assuming that $a\condi b\in\ult$ we have to show that
    \[
    (\forall Z\subseteq\varphi(a))\,T_A(\ult,Z)\subseteq\varphi(b)\,.
    \]
    If $Z\subseteq\varphi(a)$ is not closed, then $T_A(\ult,Z)=\emptyset$. So assume $Z=\varphi(\fil)$. Since $a\in\fil$ by assumption, we have by Lemma~\ref{lem:existence-for-representation} that $b\in\D{\ult}{\fil}$. So $\varphi(\D{\ult}{\fil})\subseteq\varphi(b)$, and $T_A(\ult,\varphi(\fil))\subseteq\varphi(b)$ by \eqref{eq:remark-T}.
\end{proof}

\subsection{Canonical extensions of the conditional operator} We will show that the operator $\TAcondi$ of $\Em(\frA)$ is indeed a~$\pi$-extension of the $\condi$ of $\frA\in\CA$ and in consequence $\Em(\frA)$ is the canonical extension of $\frA$. 
We will also prove that the $\pi$ and $\sigma$-extensions of $\condi$ are different. 

Let us recall that given a Boolean algebra $A$ with an $n$-ary operator $f\colon A^n\to A$ isotone at every projection $i\leqslant n$, its $\pi$-extension in the canonical extension $A^\sigma$ of $A$ is constructed along the following way. 
Let $o^n$ be the set of all sequences of length $n$ of open elements of $A^\sigma$ (i.e., those that are suprema of subsets of $A$). We will denote the elements of $A^n$ with the letter $\beta$. Then for every $\gamma\in(A^\sigma)^n$
\[
f^\pi(\gamma)\defeq
\bigwedge_{\alpha\in(\upop\gamma)\cap o^n}\bigvee_{\beta\leq\alpha}f(\beta)\,.\footnotemark
\]
\footnotetext{We refer the reader to \citep{Jonsson-et-al_BAWOI} for details.}
Since $\condi$ is an operator that is antitone in the first coordinate and isotone in the second, we will consider $A^\partial$, the dual algebra of $A$ (i.e., $A$ with the order reversed) and look at $\condi$ as the operator of the type $A^\partial\times A\to A$, and make use of the fact that $\left(A^\partial\right)^\sigma=(A^\sigma)^\partial$.

\begin{lemma}
    Let $\langle A,\condi\rangle\in \CA$ and let $\ult\in \Ul(A)$. If $Y\in \closed(\topos)$ and $O\in\topos$, then\/\textup{:}
    $\ult \in Y \TAcondi O$ iff there exist $a,b\in A$ such that $Y\subseteq \varphi(a)$, $\varphi(b)\subseteq O$ and $\ult\in \varphi(a)\TAcondi\varphi(b)$.
\end{lemma}

\begin{proof}
    Assume that $\ult\in Y \TAcondi O$. Let $\fil_Y\in \Fi(A)$ be the filter such that $Y=\varphi(\fil_Y)$. Consider the ideal $\ide_O\defeq\{a\in A: \varphi(a)\subseteq O\}$. Since $O$ is open, we have that ($\dagger$) $O=\bigcup_{a\in\ide_O}\varphi(a)$. We will prove that $\D{\ult}{\fil_Y}\cap\ide_O\neq \emptyset$. To this end, observe that any ultrafilter $\ultV$ extending $\D{\ult}{\fil_Y}$ intersects $\ide_O$. Indeed, if $\D{\ult}{\fil_Y}\subseteq\ultV$, then by the assumption and by \eqref{df:T_A} we get that $\ultV\in T_A(\ult,Y)\subseteq O$. In consequence from $(\dagger)$ it follows that $\ultV\cap \ide_O\neq\emptyset$. Thus, it cannot be the case that $\D{\ult}{\fil_Y}\cap\ide_O=\emptyset$, and so, there exists $b\in A$ such that $b\in\D{\ult}{\fil_Y}\cap\ide_O$. So we have that $\varphi(b)\subseteq O$ and there is an $a\in \fil_Y$ such that $a\rightarrowtriangle b\in \ult$. By Theorem \ref{th:RT-for-CCA}, it is the case that $\ult\in \varphi(a)\TAcondi \varphi(b)$.

    The converse implication follows immediately from \eqref{eq:cond-order-preserving} and \eqref{eq:cond-order-reversing}. 
\end{proof}

Using the fact that in the case $Y$ is not a closed subset of the Stone space of $A$, for any ultrafilter $\ult$, the set $T_A(\ult,Y)$ is empty, we obtain
\begin{proposition}
    For any subsets $U,V$ of ultrafilters of a conditional algebra $\klam{A,\condi}$\/\textup{:} $\ult\in U\TAcondi V$ iff for all closed subsets $Y$ of $U$, $T_A(\ult,Y)\subseteq V$.
\end{proposition}

\begin{lemma}
    Let $\klam{A,\condi}\in \CA$, $\ult\in \Ul(A)$, and $U,V\in \power(\Ul (A))$. Then, $\ult \in U \TAcondi V$ iff for all $Y\in \closed(\topos)$ and for all $O\in\topos$, if $Y\subseteq U$ and $V\subseteq O$, then $\ult\in Y \TAcondi O$.
\end{lemma}

\begin{proof}
    ($\Rarrow$) Suppose that $\ult\in U \TAcondi V$. Let $Y\subseteq U$ and $V\subseteq O$. It follows from \eqref{eq:cond-order-preserving} and \eqref{eq:cond-order-reversing} that $\ult\in Y \TAcondi O$. 

    \smallskip
    
    ($\Larrow$) For the reverse implication, suppose that for all $Y\in \closed(\topos)$ and for all $O\in\topos$, if $Y\subseteq U$ and $V\subseteq O$, then $\ult\in Y \TAcondi O$. Assume that $\ult\notin U \TAcondi V$, i.e, there exists $Y\in \closed(\topos)$ such that $Y\subseteq U$ and $T_A(\ult,Y)\nsubseteq V$. Pick a $\ultV\in T_A(\ult,Y)$ such that $\ultV\notin V$. Clearly, $\{\ultV\}^c$ is an open set and $V\subseteq \{\ultV\}^c$. Thus, by the assumption, $\ult\in Y\TAcondi\{\ultV\}^c$ and so $\ultV\in T_A(\ult,Y)\subseteq \{\ultV\}^c$, which is a contradiction. Therefore $\ult\in U \TAcondi V$.
\end{proof}

Let $\downop^{\closed} U$ be the set of all closed subsets of $U$, and $\upop^{\topos} V$ the set of all open supsets of $V$. Let $\Pi\defeq\downop^{\closed} U\times \upop^{\topos} V$.  As a consequence of the previous theorems, we get that 
\begin{equation}\label{eq:pi-extension-of-condi}
   U\TAcondi V=\bigcap_{\klam{Y,O}\in\Pi}\left\{\bigcup_{\klam{a,b}\in \fil_Y\times\ide_O}\varphi(a\condi b)\right\}= U\rightarrowtriangle^{\pi}V\,.
\end{equation}

\medskip

\noindent where for every $Y\in \closed(\topos)$, $\fil_Y\subseteq A$ is the filter that satisfies $\varphi(\fil_Y)=Y$; and for every $O\in\topos$, $\ide_O$ is the ideal such that $O=\bigcup_{a\in\ide_O}\varphi(a)$.

As a corollary we obtain
\begin{theorem}\label{th:pi-extension}
    If $\frA\defeq\klam{A,\condi}$ is a conditional algebra, then the operation $\TAcondi$ of $\Em(\frA)$ is a $\pi$-extension of $\condi$, and so---by Proposition~\ref{prop:power-set-algebra}---the variety $\CA$ is closed under canonical extensions.
\end{theorem}


Let us now turn to---dual to the $\pi$-extension---the notion of a \emph{$\sigma$-extension} of an operator. In an abstract setting, we define it as
\[
f^\sigma(\gamma)\defeq\bigvee_{\alpha\in(\downop\gamma)\cap c^n}\bigwedge_{\alpha\leq\beta}f(\beta)\,.
\]
where this time $c^n$ is the Cartesian $n$-th power of the closed elements of $A^\sigma$ and $\beta$ is used to denote elements from $A^n$.

\enlargethispage{1cm}

In the particular case of the conditional studied here, in the first stage, we define it for pairs of open and closed subsets of the Stone space as
\[Z^c\condi^\sigma Y=\bigcap_{\klam{a,b}\in \ide_Z\times \fil_Y} \varphi(a\condi b)\]
where $\fil_Y\in \Fi(A)$ and $\ide_Z\in \Id(A)$ are such that $Z=\varphi(\neg \ide)$ and $Y=\varphi(\fil)$. In the second stage, we extend the definition to arbitrary pairs of sets of ultrafilters via
\[U\condi^\sigma V=\bigcup_{\klam{Z^c,Y}\in\Pi} Z^c\condi^\sigma Y\,,\]
where $\Pi\defeq\upop^\sigma U \times \downop^{\closed} V$. 

On the other hand, it is a consequence of the results from \citep{Menchon-PhD,Celani-et-al-MDS} that for  a conditional algebra $\frA\defeq\langle A,\condi\rangle$, the  dual relation $G_\condi$ of the operator $\condi$ that corresponds to its $\sigma$-extension is a set of triples in $\Ul(A)\times \power(\Ul(A))\times \power(\Ul(A))$ such that 
\begin{align*}
G_\condi(\ult,Z,Y)\quad\text{iff}\quad(\exists\ide\in\Id(A))(\exists\fil\in\Fi(A))&\,(Z=\varphi(\neg \ide)\ \text{and}\ \\ 
Y=\varphi(\fil)&\ \text{and}\ (\ide\times\fil)\subseteq (\condi)^{-1}[\ult]).
\end{align*}
Let $G_\condi(\ult)\defeq\{\klam{Z,Y}\mid G_\condi(u,Z,Y)\}$. In \citep{Menchon-PhD,Celani-et-al-MDS} it was proved that 
\[\ult\in U \condi^\sigma V\quad \text{iff}\quad (\exists\klam{Z, Y} \in G_\condi(\ult))\, (Z \cap U) \cup (Y \cap V
^c) = \emptyset\,.\]
Using the equivalence, we are going to show that
\begin{theorem}
If $\frA\defeq\klam{A,\condi}\in\CA$, then $\klam{\power(\Ul(A)),\condi^\sigma}\in\CA$, i.e., the variety of conditional algebras is closed under $\sigma$-extensions.
\end{theorem}
\begin{proof}
\eqref{C1} Let $U\subseteq \Ul(A)$. We will show that $U\condi^\sigma \Ul(A)=\Ul(A)$. Let $\ult\in \Ul(A)$. By \eqref{C1}, $A\times \{1\}\subseteq (\condi)^{-1}[\ult]$. Then, $G_\condi(\ult,\emptyset,\Ul(A))$ and it follows that $\ult\in U\condi^\sigma \Ul(A)$.

\smallskip

\eqref{C2} Let $U, V,W\subseteq \Ul(A)$. Suppose that $\ult\in (U\condi^\sigma V)\cap (U\condi^\sigma W)$. Then, there exists a pair $\klam{Z_1,Y_1}\in G_\condi(\ult)$ such that $(Z_1 \cap U) \cup (Y_1 \cap V^c) = \emptyset$ and there exists a pair $\klam{Z_2,Y_2}\in G_\condi(\ult)$ such that $(Z_2 \cap U) \cup (Y_2 \cap W^c) = \emptyset$. Let $\ide_1,\ide_2\in \Id(A)$ and $\fil_1,\fil_2\in \Fi(A)$ be such that $\varphi(\neg \ide_1)=Z_1$, $\varphi(\neg \ide_2)=Z_2$, $\varphi(\fil_1)=Y_1$ and $\varphi(\fil_2)=Y_2$. Let $\ide_3=\ide_1\cap \ide_2$ and let $\fil_3$ be the filter generated by $\fil_1\cup \fil_2$. First, we will show that $\ide_3\times \fil_3\subseteq (\condi)^{-1}[\ult]$. Let $a\in \ide_3$ and $b\in \fil_3$. So, there exist $c\in \fil_1$ and $d\in \fil_2$ such that $c\wedge d\leq b$. Then, $a\condi (c\wedge d)\leq a\condi b$. By \eqref{C2}, $a\condi (c\wedge d)=(a\condi c) \wedge (a\condi d)$ and by assumption $a\condi c\in \ide_1\times \fil_1\subseteq (\condi)^{-1}[\ult]$, $a\condi d\in \ide_2\times \fil_2\subseteq (\condi)^{-1}[\ult]$. It follows that $(a\condi c)\wedge (a\condi d)\in \ult$ and thus $a\condi b\in \ult$. So, $G_\condi(\ult,\varphi(\neg \ide_3),\varphi(\fil_3))$. It is easy to see that $U\cap \varphi(\neg \ide_3)=\emptyset$ and $\varphi(\fil_3)\subseteq V\cap W$. Therefore, $\ult\in U\condi (V\cap W)$. The converse direction follows immediately.

\smallskip

\eqref{C3} Let $U, V,W\subseteq \Ul(A)$. Suppose that $\ult\in (U\cup V)\condi^\sigma W$. Then, there exists a pair $\klam{Z,Y}\in G_\condi(\ult)$ such that $(Z \cap (U\cup V)) \cup (Y \cap W^c) = \emptyset$. It is easy to see that  $(Z \cap U) \cup (Y \cap W^c) = \emptyset$ and   $(Z \cap  V) \cup (Y \cap W^c) = \emptyset$. Therefore $u\in (U\condi^\sigma W)\cap (V\condi^\sigma W)$.
\end{proof}

It is a consequence of general properties of $\sigma$ and $\pi$-extensions that for any $U,V\in\power(\Ul(A))$
\[
U\condi^\sigma V\subseteq U\condi^\pi\pi V\,.
\]
The operator $\condi$ is \emph{smooth} if the other inclusion holds as well.\footnote{For details on smooth operators, see e.g., \citep{Gehrke-CEESAUM}.}

\begin{example}
    Let us show that there is a subvariety of $\CA$ in which $\condi$ is a smooth operator. To this end consider any Boolean algebra $A$ with a binary operator $\condi$ such that ($\dagger$) $a\condi b\defeq b$. It is clear that $\frA\defeq\klam{A,\condi}\in\CA$. For this algebra, we have $\D{\ult}{\fil}=\ult$, for any ultrafilter $\ult$ and any filter $\fil$. Indeed, if $a\in\D{\ult}{\fil}$, then there is $b\in\fil$ such that $b\condi a\in\ult$. By ($\dagger$), $a\in\ult$. If, on the other hand, $a\in\ult$, then $\one\condi a\in\ult$ and $\one\in\fil$, so $a\in\D{\ult}{\fil}$. So for any $\ult$ and $\fil$: $T_A(\ult,\varphi(\fil),\ult)$.
    
    Thanks to this we can prove that $U\TAcondi V=V$. Indeed, if $\ult\in U\TAcondi V$, then for any $Y\subseteq U$, $T_A(\ult,Y)\subseteq U$. In particular, $T_A(u,\emptyset)\subseteq V$, and $\ult\in T(\ult,\emptyset)$. Thus, $\ult\in V$. For the same reason, if $u\in V$, then $u\in U\TAcondi V$.

    To show $\condi$ is smooth, we have to prove that $U\condi^\pi V\subseteq U\condi^\sigma V$. Thus we assume that $\ult\in U\condi^\pi V$, which by Theorem~\ref{th:pi-extension} is equivalent to $\ult\in U\TAcondi V$. Fix $\ide\defeq A$ and $\fil\defeq\ult$. Thus $\ide\times\fil\subseteq(\condi)^{-1}[u]$. Further
    \[
        \varphi(\neg\ide)^c=\varphi(\neg A)^c=\varphi(A)^c=\emptyset^c=\Ul(A)\supseteq U\,.
    \]
    By the first step of the proof and by the assumption we have that $u\in V$, thus $\varphi(\fil)=\{u\}\subseteq V$. In consequence $u\in U\condi^\sigma V$, as required.
\end{example}

However, in general, we have
\begin{theorem}
    $\condi$ is not smooth in $\CA$.
\end{theorem}
\begin{proof}[Proof by example]
Let us consider the Boolean algebra $\power(\Nat)$. We will denote the set of even numbers by $\Even$. Let us consider a binary operator $\condi$ on $\power(\Nat)$ defined by
\[
P\condi Q\defeq\begin{cases}
    \Nat & \text{ if }P\subseteq Q,\\
    \Even& \text{ if }P\nsubseteq Q\text{ and }\Even\nsubseteq P,\\
    \emptyset & \text{ otherwise.}
\end{cases}
\]
It is obvious that for any set $P$, $P\condi P=\Nat$. In consequence, for any $\ult\in\Ul(\power(\Nat))$ and any closed subset $Y$ of $\Ul(\power(\Nat))$, $\fil_Y\subseteq\D{\ult}{\fil_Y}$, and so $T_{\power(\Nat)}(\ult,Y)\subseteq Y$. Therefore $P\condi_{T_{\power(\Nat)}}P=\Ul(\power(\Nat))$, and by Theorem~\ref{th:pi-extension} it is the case that $P\condi^\pi P=\Ul(\power(\Nat))$.
 We will show that for the set $E\defeq\{\upop \{e\}:e\in \Even\}$, $E\condi^\sigma E\neq \Ul(\power(\Nat))$. In particular, we will prove that $\upop \{3\}\notin E\condi^\sigma E$. Suppose, to get a contradiction, that $\upop \{3\}\in E\condi^\sigma E$. Then, there are an ideal $\ide$ and a~filter $\fil$ of $\power(\Nat)$ such that
\begin{equation}\tag{$\dagger$}\label{eq:for-non-smooth}
\ide\times\fil\subseteq(\condi)^{-1}[\upop \{3\}]\tand (\varphi(\neg \ide) \cap E) \cup (\varphi(\fil) \cap E^c) = \emptyset\,.
\end{equation}

From the first conjunct of \eqref{eq:for-non-smooth}, we obtain that for all $P\in\ide$ and all $Q\in\fil$, $P\subseteq Q$. Indeed, pick an arbitrary pair $\klam{P,Q}\in\ide\times\fil$. It follows that $P\condi Q$ is in $\upop\{3\}$, that is $3\in P\condi Q$. By the definition of the operator, $P\condi Q=\Nat$ and it follows that $P\subseteq Q$.

From the second conjunct of \eqref{eq:for-non-smooth}, we get that $\varphi(\fil)\subseteq E\subseteq \varphi(\neg \ide)^c$ and it follows that $\fil\cap \ide\neq \emptyset$. So, from this and the previous paragraph stems the existence of $C\in \fil\cap\ide$ such that $\ide=\downop C$ and $\fil= \upop C$. Thus $\varphi(\upop C)\subseteq E\subseteq \varphi(\neg \downop C)^c=\varphi(\upop C)$, and so $E=\varphi(\upop C)$. But this is impossible because $E$ is not closed in $\Ul(\power(\Nat))$.

To conclude, $E\condi^\sigma E\neq \Ul(\power(\Nat))= E\condi^\pi E$.
\end{proof}

\section{Conditional algebras and multi-modal antitone algebras}

In the introduction, we observed after Chellas that there is a strong connection between the conditional binary operator and the relativized necessity operators. Let us study this connection in the algebraic setting and with the help of the results obtained so far.

\begin{definition}
    Suppose $A$ is a Boolean algebra and $B$ is 
    a~subalgebra of~$A$. By a \emph{multi-modal antitone algebra} we understand an algebra $\frM\defeq\langle A,\{\Box_b\}_{b\in B}\rangle$ such that for every $b\in B$, $\Box_b$ is a unary necessity operator satisfying
    \begin{align}
        \Box_b(\one)&{}=\one\,,\tag{M1}\label{M1}\\
        \Box_b(a\wedge c)&{}=\Box_b(a)\wedge\Box_b(c)\,,\tag{M2}\label{M2}
    \end{align}
    and 
    \begin{equation}
    \Box_{b_1\vee b_2}(a)\leq\Box_{b_1}(a)\wedge\Box_{b_2}(a)\,.\tag{M3}\label{M3}
\end{equation}
A multimodal antitone algebra $\frM\defeq\langle A,\{\Box_b\}_{b\in B}\rangle$ is \emph{full} if $B=A$. Let $\MMA$ be the class of all multi-modal antitone algebras.\footnote{We have not been able to determine if such algebras have been analyzed before. \cite{Duntsch-et-al-BAAFIS} in the context of information systems define somewhat similar algebras, each with a family of operators indexed by all subsets of a fixed set. They call such structures \emph{Boolean algebras with relative operators}. In such a case, the index set is always a complete BA, and \emph{a priori} does not have to be related to the domain of the algebra.}
\end{definition}
It is not hard to see that \eqref{M3} is equivalent to
    \begin{equation*}
        b_1\leq b_2\rarrow \Box_{b_2}(a)\leq\Box_{b_1}(a)\,.\tag{M3$^\ast$}\label{M3ast}
    \end{equation*}
Thus, the name `antitone' comes from the third axiom that postulates an order on the set of operators that is reversed with respect to the Boolean order on~$B$.

\begin{theorem}\label{th:CA-MMA-terme-equivalent}
    The classes of conditional algebras and full multi-modal antitone algebras are definitionally equivalent.
\end{theorem}
\begin{proof}
    Every conditional algebra $\langle A,\condi\rangle$ can be expanded to  $\langle A,\condi,\{\Box_a\}_{a\in A}\rangle$ by means of the following definition
    \[
    \Box_a(b)\defeq a\condi b\,.
    \]
    From the results above it follows that every $\Box_a$ is a~necessity operator and that the expansion satisfies \eqref{M3}.

    \smallskip

    If we start with the class of full MMAs, then every element of it can be expanded to the algebra $\langle A,\condi,\{\Box_a\}_{a\in A}\rangle$ via
    \[
        a\condi b\defeq\Box_a(b)\,.
    \]
    It is routine to check that the conditional operator satisfies the axioms \eqref{C1}--\eqref{C3}.
\end{proof}

    \pagebreak

We will show now that 
\begin{theorem}\label{th:MMA-closed-for-CE}
The class of multi-modal antitone algebras is closed under canonical extensions. 
\end{theorem}

For a fixed algebra $\frM\in\MMA$, its ultrafilter frame is the structure
\[
\Uf(\frM)\defeq\langle\Ul(A),\{Q_b\}_{b\in B}\rangle 
\]
such that for every $b\in B$
        \[
            Q_b(u,v)\iffdef\Box_{b}^{-1}[u]\subseteq v\,.
        \]
Let
\[
Q_b(u)\defeq\{v\in\Ul(A)\mid Q_b(u,v)\}\,.
\]
\begin{lemma}\label{lem:order-on-Qb}
    For any $b,c\in B$ and $u\in\Ul(A)$ we have
    \[
        b\leq c\rarrow Q_b(u)\subseteq Q_c(u)\,.
    \]
\end{lemma}
\begin{proof}
    Let $b\leq c$. If $v\in Q_b(u)$, then $\Box_b^{-1}[u]\subseteq v$. So, if $d\in \Box_c^{-1}[u]$, then $\Box_c(d)\in u$, and so $\Box_b(d)\in u$, i.e., $d\in\Box_b^{-1}[u]$, and thus $d\in v$, as required.
\end{proof}

The canonical extension of $\frM$ is the complex algebra of $\Uf(\frM)$ 
\[\Em(\frM)\defeq\Cm(\Uf(\frM))=\langle\power(\Ul(A)),\{[Q_b]\}_{b\in B}\rangle\,\] 
where
\[
[Q_b]\colon\power(\Ul(A))\to \power(\Ul(A)) 
\]
is standardly defined as
\[
[Q_b](X)\defeq\{u\in\Ul(A)\mid Q_b(u)\subseteq X\}\,.
\]

\begin{lemma}\label{lem:Box-to-Q-injection}
    For any $b,c\in B$, if $\Box_b\neq\Box_c$, then $[Q_b]\neq [Q_c]$.
\end{lemma}
\begin{proof}
Suppose there is an $x\in A$ such that $\Box_b(x)\neq\Box_c(x)$. Without the loss of generality assume that $\Box_b(x)\nleq\Box_c(x)$, i.e., $\Box_b(x)\cdot\neg\Box_c(x)\neq\zero$. Let $u$ be an ultrafilter that contains $\Box_b(x)\cdot\neg\Box_c(x)$. In consequence,
\[
x\in\Box_b^{-1}[u]\quad\text{and}\quad x\notin\Box_c^{-1}[u]\,.
\]
Since $\Box_c^{-1}[u]$ is a filter that does not have $x$ among its elements, there is an ultrafilter $w$ such that $\Box_c^{-1}[u]\cup\{\neg x\}\subseteq w$. Therefore, $w\notin\varphi(\Box_b^{-1}[u])$, and so $Q_c(u)\nsubseteq\varphi(\Box_b^{-1}[u])$. So $u\notin[Q_c](\varphi(\Box_b^{-1}[u]))$, but $u\in [Q_b](\varphi(\Box_b^{-1}[u]))$, and therefore $[Q_b]\neq[Q_c]$, as required.
\end{proof}

\begin{lemma}\label{lem:MMA-embedding}
    The Stone mapping $\varphi\colon A\to\power(\Ul(A))$ is an embedding of $\frM$ into $\Em(\frM)$, i.e.,
        \[\varphi(\Box_b(a))=[Q_b](\varphi(a))\,.\]    
\end{lemma}
\begin{proof}
        ($\subseteq$) If $u\in\varphi(\Box_b(a))$, then $a\in\Box_b^{-1}[u]$, so every ultrafilter $v$ that extends $\Box_b^{-1}[u]$ contains $a$. In other words, for every ultrafilter $v$, if $Q_b(u,v)$, then $v\in\varphi(a)$. Thus $u\in[Q_b](\varphi(a))$.

        \smallskip

        ($\supseteq$) If $u\in[Q_b](\varphi(a))$, then for every $v$ such that $\Box_b^{-1}[u]\subseteq v$, $v$ contains $a$, which entails that $a$ must be in $\Box_b^{-1}[u]$. Thus $u\in\varphi(\Box_b(a))$ as required.
\end{proof}

\begin{proof}[Proof of Theorem~\ref{th:MMA-closed-for-CE}]
    Fix $\frM\defeq\langle A,\{\Box_b\}_{b\in B}\rangle\in\MMA$ and take its complex algebra $\Em(\frM)$. From the lemma above it follows that both algebras are of the same type.  By Lemma~\ref{lem:MMA-embedding}, 
    $B$~is isomorphic to $\varphi[B]$, the image of $B$ in $\power(\Ul(A))$ via the standard Stone embedding. That each $[Q_b]$ is a unary necessity operator on $\power(\Ul(A))$ follows from the standard definition of $[Q_b]$ above. We also have that
    \[
    b_1\leq b_2\rarrow [Q_{b_2}](X)\subseteq [Q_{b_1}](X)\,.
    \]
    Indeed, assuming that $b_1\leq b_2$ and taking $u\in [Q_{b_2}](X)$, we have that $Q_{b_2}(u)\subseteq X$, and from Lemma~\ref{lem:order-on-Qb} we get that $Q_{b_1}(u)\subseteq Q_{b_2}(u)$. Thus $u\in[Q_{b_1}](X)$. So $\Em(\frM)$ is a multi-modal antitone algebra.
\end{proof}

By Theorem~\ref{th:CA-MMA-terme-equivalent} the classes of conditional algebras and full multi-modal antitone algebras are term equivalent. In Theorem \ref{th:MMA-closed-for-CE} we have proven that the canonical extension of any $\frM\in\MMA$ is also in the class. However, canonical extensions of multi-modal algebras do not have to be full. On the other hand, for the canonical extension $\klam{\power(\Ul(A)),\condi_{T_A}}$ of $\frA\in\CA$,  there exists a full multi-modal term-equivalent algebra $\klam{\power(\Ul(A)),\{\Box_U\}_{U\in\power(\Ul(A))}}$. So, if $\klam{A,\{\Box_a\}_{a\in A}}$ is obtained from $\klam{A,\condi}$ which is infinite, its canonical extension
\[
\klam{\power(\Ul(A)),\{[Q_a]\}_{a\in A}}
\]
is a multi-modal algebra with $\{[Q_a]\}_{a\in A}\subseteq\{\Box_U\}_{U\in\power(\Ul(A))}$ (but not necessarily vice versa). However, we will show that for each $a\in A$, $[Q_a]=\Box_{\varphi(a)}$, and every operator $\Box_U$ can be characterized by means of a family of $[Q_a]$'s.

To begin, let us observe that the operation $D_{\filH}^{\condi}(\fil)$ can be characterized utilizing the relation between the conditional operator and unary necessity operators of full MMAs.  Since $\Box_a$ is a necessity operator, $\Box_a^{-1}[\fil]=\{b\in A: a\condi b\in \fil\}$ is a filter, provided $\fil$ is a filter too. Also, as $a$ is in $\upop a$, it must be the case that
\[
\Box_a^{-1}[\fil]\subseteq D_{\fil}^{\condi}(\upop a)\,.
\]
Further, we have that
\begin{lemma}\label{lem:Box-and-D}
    If $\frA\defeq\langle A,\condi\rangle$ is a Boolean algebra with a binary operator $\condi$, then the algebra satisfies \eqref{C3} iff for every $a\in A$ and every $\fil\in\Fi(A)$, $D_{\fil}^{\condi}(\upop a)\subseteq \Box_a^{-1}[\fil]$. Thus, if $\frA\in\CA$, then 
    \[
        \Box_a^{-1}[\fil]=D_{\fil}^{\condi}(\upop a)\,.
    \]
\end{lemma}
\begin{proof}
The implication from left to right stems from \eqref{eq:cond-order-reversing}, which is a consequence of \eqref{C3}. 

For the other direction, suppose that there exist $a,b,c\in A$ such that $(a\vee b)\condi c\nleq (a\condi c)\wedge (b\condi c)$. Then, there exists $\fil\in \Fi(A)$ such that $(a\vee b)\condi c\in \fil$ and $(a\condi c)\wedge (b\condi c)\notin \fil$. Assume without the loss of generality that $a\condi c\notin \fil$, i.e., $c\notin\Box_a^{-1}[\fil]$. But $c\in D_{\fil}^{\condi}(\upop a)$, as $a\leq a\vee b$.
\end{proof}

\begin{lemma}\label{lem:props-for-Box}
    If $\frA\defeq\langle A,\condi\rangle\in \CA$, $a\in A$ and $\filH,\fil\in\Fi(A)$, then
    \[
        D^{\condi}_{\fil}(\filH)=\bigcup_{a\in \filH} D^{\condi}_{\fil}(\upop a)=\bigcup_{a\in \filH}\Box^{-1}_a[\fil]\,.
    \]
\end{lemma}
\begin{proof}
A for all $a\in \filH$, $\upop a\subseteq\filH$, so by monotonicity of $D^{\condi}_{\fil}$ we obtain that $D^{\condi}_{\fil}(\upop a)\subseteq D^{\condi}_{\fil}(\filH)$. On the other hand, if $b\in D^{\condi}_{\fil}(\filH)$, then there exists $a\in \filH$ such that $a\condi b\in \fil$. It follows that $b\in D^{\condi}_{\fil}(\upop a)\subseteq \bigcup_{a\in \filH } D^{\condi}_{\fil}(\upop a)$.
\end{proof}

Since
\[
D^{\condi}_{\ult}(\fil)\subseteq \ultV\quad\text{iff}\quad \text{for all $a\in \fil$},\ \Box_a^{-1}[\ult]=D^{\condi}_{\ult}(\upop a) \subseteq \ultV  
\]
we get that 
\begin{equation}\label{eq:for-th-CA-MMA}
D^{\condi}_{\ult}(\fil)\subseteq \ultV\quad\text{iff}\quad\ultV\in \bigcap_{a\in \fil} Q_a(\ult)\quad\text{iff}\quad(\forall a\in\fil)\,Q_a(\ult,\ultV)\,.
\end{equation}

\begin{theorem}
Suppose $\frA\defeq\klam{A,\condi}\in\CA$ and let $\frM_{\frA}\defeq\klam{A,\{\Box_a\}_{a\in A}}$ be its definitionally equivalent multi-modal algebra. Let
\[\frM_{\Em(\frA)}\defeq\klam{\power(\Ul(A)),\{\Box_U\}_{U\in\power(\Ul(A))}}\] be the full multi-modal algebra definitionally equivalent to the canonical extension of $\frA$. Let $\Em(\frM)=\klam{\power(\Ul(A)),\{[Q_a]\}_{a\in A}}$ be the canonical extension of $\frM$. Then\/\textup{:}
    \begin{enumerate}\itemsep+2pt
        \item for every $a\in A$, $[Q_a]=\Box_{\varphi(a)}$,
        \item for all $U,V\in\power(\Ul(A))$
        \begin{equation*}
   \Box_U(V)=\bigcap_{\klam{Y,O}\in\Pi}\left\{\bigcup_{\klam{a,b}\in \fil_Y\times\ide_O}\Box_{\varphi(a)}(\varphi(b))\right\}=\bigcap_{\klam{Y,O}\in\Pi}\left\{\bigcup_{a\in \fil_Y}[Q_a](O)\right\}\,,
\end{equation*}
        where $\Pi\defeq\downop^{\closed} U\times \upop^{\topos} V$.     
\end{enumerate}
\end{theorem}
\begin{proof}
(1) We will show that for any $V$, $[Q_a](V)=\Box_{\varphi(a)}(V)$. 

For the left-to-right inclusion, let $\ult\in[Q_a](V)$, i.e., $Q_a(\ult)\subseteq V$. Let $Y$ be an arbitrary closed subset of $\varphi(a)$, let $\ultV\in T_A(\ult,Y)$. In consequence, $\D{\ult}{\fil_Y}\subseteq\ultV$. But $\upop a\subseteq\fil_Y$, and $\D{u}{\upop a}\subseteq\D{u}{\fil_Y}$, therefore $\D{u}{\upop a}\subseteq\ultV$. So from \eqref{eq:for-th-CA-MMA} we obtain that $Q_a(\ult,\ultV)$, i.e., $\ultV\in Q_a(\ult)$. Thus $\ultV\in\varphi(a)\TAcondi V=\Box_{\varphi(a)}(V)$. 

For the right-to-left inclusion, assume that $\ult\in\Box_{\varphi(a)}(V)=\varphi(a)\TAcondi V$. To prove that $\ult\in[Q_a](V)$ we need to show that $Q_a(\ult)\subseteq V$. Thus, pick $\ultV\in Q_a(\ult)$, i.e., $Q_a(\ult,\ultV)$. So $\Box_a^{-1}[\ult]\subseteq\ultV$ and by Lemma~\ref{lem:Box-and-D} we get that $\D{\ult}{\uparrow a}\subseteq v$. In consequence $\ultV\in T_A(\ult,\varphi(a))$, and from the assumption it follows that $v\in V$, as required.

\smallskip

(2) This is a consequence of the first point and \eqref{eq:pi-extension-of-condi}.
\end{proof}

\section{A representation theorem for t-frames}

In general, it is not possible to show that every t-frame $\frX\defeq\langle X,T\rangle$ can be embedded into the ultrafilter frame of its full complex algebra via the standard embedding  $e\colon X\to\Ul(\power(X))$ defined by $e(x)\defeq\{M\in\power(X)\mid x\in M\}$ (that works in the case of frames with elementary relations). The following example demonstrates this fact.

\begin{example}
    Take $\frN\defeq\langle\Nat,T\rangle$ to be a t-frame with the domain of natural numbers and $T$ such that
    \[
    T(n,Z,m)\iffdef n+m\in Z\,.
    \]
    Clearly, any triple $\klam{n,\Nat,m}$ is in~$T$. However, for   $\langle\Ul(\power(\Nat)),T_{\power(\Nat)}\rangle$ it cannot be the case that $T_{\power(\Nat)}(e(n),e[\Nat],e(m))$, precisely because $e[\Nat]$ is the family of all principal ultrafilters of $\power(\Nat)$, and for this family there is no filter $\fil$ such that $e[\Nat]=\varphi(\fil)$. Indeed, such a filter would have to be a subset of $\bigcap e[\Nat]$, which in this case is just $\{\Nat\}$, and so $\varphi(\fil)=\Ul(\power(X))$. Observe that $e[\Nat]$---the set of all isolated points of the space $\Ult(\power(X))$---is not closed in this space.\QED

\end{example}

In light of the example, $T\subseteq X\times\power(X)\times X$ may be, in a sense, <<too large>>. The middle projection will usually contain sets $Z$ whose images $e[Z]$ are not closed in the Stone space $\Ul(\power(X))$, and in consequence are not elements of the second projection of $T_{\power(X)}$. Therefore, to guarantee that $e$ embeds $T$ into a frame based on the set of ultrafilters, we should find a relation that extends $T_{\power(X)}$ in such a way so as to accommodate all the required triples from~$T$. Below, such a relation is defined and the set-theoretical representation theorem for this relation is proved.\footnote{The idea for the relation and most of the proofs below are due to one of the referees. We express our utmost gratitude for pointing to us a way towards the set-theoretical representation for t-frames.}

\begin{definition}
Let $\frX\defeq\langle X,T\rangle$ be a t-frame and let $\calF_X$ be a family of subsets of $X$. $T$~is \emph{upward closed in $\calF_x$ in the second coordinate} if $T(x,Z,y)$ and $Z\subseteq U\in\calF_X$ entail that $T(x,U,y)$.
\end{definition}

Obviously, any relation $T\subseteq X\times\power(X)\times X$ can be extended to the relation $\barT$ which is closed in the second coordinate via
\[
\barT(x,U,y)\iffdef(\exists Z\in\power(U))\,T(x,Z,y)\,.
\]
Let us call $\barT$ the \emph{upward closure} of $T$.

\begin{lemma}\label{lem:simpler-Tcondi}
If $\frX\defeq\langle X,T\rangle$ is a t-frame and  $\calF_X$ is a family of subsets of $X$ such that $T$ is upward closed in $\calF_X$ in the second coordinate, then for the conditional algebra $\langle \calF_X,\Tcondi \rangle$ we have 
\begin{equation*}
U\Tcondi V=\{x\in X:T(x,U)\subseteq V\}\,.  
\end{equation*}
\end{lemma}
\begin{proof}
    Clearly, if $T(x,U)\subseteq V$ and $Z\subseteq U$ is such that $T(x,Z,y)$, then the closure condition entails that $T(x,U,y)$ and so $y\in V$. The other inclusion is a direct consequence of \eqref{df:Tcondi}.
\end{proof}

\begin{lemma}\label{lem:the-same-complex-algebras}
    If $\frX\defeq\langle X,T\rangle$ is a t-frame, then $\frX$ and $\overline{\frX}\defeq\langle X,\barT\rangle$ have the same full complex algebras.
\end{lemma}
\begin{proof}
    We have to show that $\mathord{\Tcondi}=\mathord{\barTcondi}$.

    ($\subseteq$) If $x\in U\Tcondi V$, then for all $Z\subseteq U$, $T(x,Z)\subseteq V$. Pick $y\in\barT(x,U)$. By the definition of $\barT$ we have that there is $Z_0\subseteq U$ such that $y\in T(x,Z_0)$. So $y\in V$, and so $\barT(x,U)\subseteq V$. By Lemma~\ref{lem:simpler-Tcondi} we obtain that $x\in U\barTcondi V$.

    \smallskip

    ($\supseteq$) If $x\in U\barTcondi V$, then $\barT(x,U)\subseteq V$. Let $Z$ be such that $Z\subseteq U$ and $y\in T(x,Z)$. Then $y\in\barT(x,Z)$, and as $\barT$ is upward closed we obtain that $y\in\barT(x,U)$. So $y\in V$, and $T(x,Z)\subseteq V$. By arbitrariness of $Z$ we have that $x\in U\Tcondi V$.
\end{proof}

For any conditional algebra $\frA\defeq\langle A,\condi\rangle$ we can define the following tenary hybrid relation in $\Ul(A)\times\power(\Ul(A))\times\Ul(A)$
\[
\textstyle T'_{A}(u,Z,v)\iffdef(\forall a\in\bigcap Z)(\forall b\in A)\,(\text{if $a\condi b\in u$, then $b\in v$})\,.
\]
Let us observe that the restriction of $T'_{A}$ to closed subsets of the space of ultrafilters in the middle component is precisely the relation $T_A$ from~\eqref{df:T_A}. Indeed, for $u,v\in\Ul(A)$, $Z\in\closed(\Ul(A))$ and $\fil\in\Fil(A)$ such that $Z=\varphi(\fil)$ we obtain
\begin{align*}
T_A(\ult,\varphi(\fil),v)&{}\qtiff D^{\condi}_{\ult}(\fil)\subseteq v\\&{}\qtiff(\forall b\in A)(\forall a\in\fil)\,(\text{if $a\condi b\in\ult$, then $b\in v$})\\
&{}\qtiff(\forall a\in\fil)(\forall b\in A)\,(\text{if $a\condi b\in\ult$, then $b\in v$})\\
&\textstyle {}\qtiff(\forall a\in\bigcap\varphi(\fil))(\forall b\in A)\,(\text{if $a\condi b\in\ult$, then $b\in v$})\\
&{}\qtiff T'_{A}(\ult,\varphi(\fil),v)\,.
\end{align*}
Further, given a t-frame $\frX\defeq\langle X,T\rangle$ with $T$ upward closed, and the frame $\langle\Ul(\power(X)),\TprimePX\rangle$ we can show that the standard mapping $e\colon X\to\ult(\power(X))$ is en embedding. Indeed, for any $x,y\in X$ and $Z\subseteq X$ we obtain (using Lemma~\ref{lem:simpler-Tcondi})
\begin{align*}
\TprimePX(e(x),e[Z],e(y))&{}\qtiff(\forall U\in\textstyle\bigcap e[Z])(\forall V\in\power(X))\,(\text{if $U\Tcondi V\in e(x)$, then $V\in e(y)$})\\
&\qtiff(\forall U\in\upop Z)(\forall V\in\power(X))\,(\text{if $x\in U\Tcondi V$, then $y\in V$})\\
&\qtiff(\forall U\in\upop Z)(\forall V\in\power(X))\,(\text{if $T(x,U)\subseteq V$, then $y\in V$})\\
&\qtiff(\forall V\in\power(X))\,(\text{if $T(x,Z)\subseteq V$, then $y\in V$})\\
&\qtiff y\in T(x,Z)\\
&\qtiff T(x,Z,y)\,.
\end{align*} 
\begin{theorem}
Every t-frame $\frX\defeq\langle X,T\rangle$ can be embedded (via the standard mapping $e\colon X\to\Ul(\power(X))$) into the frame 
\[
\mathfrak{Y}\defeq\left\langle\Ul(\power(X)),\barTprimePX\right\rangle 
\]
whose relation $\barTprimePX$ restricted to $\Ul(\power(X))\times\closed(\Ul(\power(X)))\times\Ul(X)$ is precisely the relation $T_{\power(X)}$ of the full complex algebra of $\frX$.
\end{theorem}
\begin{proof}
Fix $\frX\defeq\langle X,T\rangle$ and take $\overline{\frX}\defeq\langle X,\barT\rangle$. The relation 
\[
\barTprimePX\subseteq\Ul(X)\times\closed(\Ul(X))\times\Ul(X)
\]
as defined above coincides with $\barT_{\power(X)}$ on triples from $\Ul(X)\times\closed(\Ul(X))\times\Ul(X)$, and by Lemma~\ref{lem:the-same-complex-algebras} coincides in the same way with $T_{\power(X)}$. Since $\barT$ is upward closed by Lemma~\ref{lem:simpler-Tcondi} we have that 
\[
U\barTcondi V=\{x\in X:\barT(x,U)\subseteq V\}\,.
\]
Thanks to this, we already know that the following equivalence obtains
\begin{align*}
\barTprimePX(e(x),e[Z],e(y))\qtiff\barT(x,Z,y)\,. 
\end{align*}
As $T\subseteq\barT$, we have that $\frX$ embeds into $\mathfrak{Y}$ via the standard mapping~$e$. 
\end{proof}

\section{Topological duality for conditional algebras}

In this section, we turn to the development of topological duality for conditional algebras and Boolean spaces expanded with a ternary relation.

A topological space $\langle X,\tau\rangle$ is \emph{Boolean} if it is zero-dimensional, compact, and Hausdorff. In the sequel, we consider triples $\langle X,\tau,T\rangle$ such that $\langle X,\tau\rangle$ is a Boolean space and $T\subseteq X\times \closed(\tau)\times X$ is a hybrid ternary relation that simultaneously satisfies the following three conditions:
\begin{enumerate}[label=(T\arabic*),itemsep=0pt]
    \item\label{T1} For every $x \in X$ and for every $Y\in\closed(\tau)$
    \[
    T(x,Y)=\{y\in X:T(x,Y,y)\}\in \closed(\tau).
    \]
    \item\label{T2} For all clopen sets $U, V$, the set $U\Tcondi V$ is clopen.
    \item\label{T3} For $Y\in\closed(\tau)$: $T(x,Y,y)$ if and only if $T(x,U,y)$ for all $U\in\clopen(\tau)$ such that $Y\subseteq U$.
\end{enumerate}
Any triple $\frX\defeq\langle X,\tau,T\rangle$ such that $\langle X,\tau\rangle$ is a Boolean space and $T$ meets \ref{T1}--\ref{T3} will be called a \emph{conditional space}. 

\begin{example}
    Let us show that the notion of the conditional space is not vacuous. To this end, let $A=\{0,a,b,1\}$ be the four-element Boolean algebra and let $\condi\colon A\times A\to A$ be the operation defined by:
    $x\condi y\defeq y$. It is easy to check that $\langle A,\condi\rangle$ is a conditional algebra. Then,  
    \begin{align*}
    T_A=\{&\klam{\upop a,\emptyset,\upop a},\klam{\upop a,\varphi(a),\upop a},\klam{\upop a,\varphi(b),\upop a},\klam{\upop a,\mathrm{Ul}(A),\upop a},\\&\klam{\upop b,\emptyset,\upop b},\klam{\upop b,\varphi(a),\upop b},\klam{\upop b,\varphi(b),\upop b},\klam{\upop b,\mathrm{Ul}(A),\upop b}\}.
    \end{align*}
In consequence,  $\langle \Ul(A),\topos,T_A\rangle$ is a conditional space.\QED
\end{example}

The proof of the following proposition is straightforward and we leave it to the reader.
\begin{proposition}\label{prop:monotonicity-for-T}
    If $\langle X,\tau,T\rangle$ is a conditional space, then $T$ is upward closed in $\closed(\tau)$, and so it is also upward closed in $\clopen(\tau)$. In consequence, by Lemma~\ref{lem:simpler-Tcondi}, for any $U,V\in\clopen(\tau)$, $U\Tcondi V=\{x\in X\mid T(x,U)\subseteq V\}$.
\end{proposition}

\begin{definition}
    If $\frA\defeq\langle A,\condi\rangle$ is a conditional algebra, then its \emph{expanded Stone space} is the triple $\Es(\frA)\defeq\langle\Ul(A),\topos,T_A\rangle$ such that $\langle\Ul(A),\topos\rangle$ is the Stone space of~$A$ and $\klam{\Ul(A),T_A}$ is the ultrafilter frame of $\frA$.
\end{definition}

\begin{theorem}\label{th:Es-is-CS}
    If $\frA\defeq\langle A,\condi\rangle\in\CA$, then $\Es(\frA)$ is a conditional space.
\end{theorem}
\begin{proof}
By \eqref{df:T_A} we have that
\[
T_A\subseteq \Ul(A)\times\closed(\topos)\times\Ul(A)\,.
\]

\ref{T1} holds by Lemma~\ref{lem:D-filter} and \eqref{eq:remark-T}. 

\smallskip

For \ref{T2} recall that $\varphi[A]$ is the algebra of clopens of the Stone space of $A$. Therefore---by Theorem \ref{th:RT-for-CCA}---for any clopen subsets $U$ and $V$ of $\Ul(A)$, $U\trarrow_{T_A} V$ is clopen. 

\smallskip

For left-to-right direction of \ref{T3}, assume that $T_A(\ult,\varphi(\fil),\ultV)$ and let $a\in A$ be such that $\varphi(\fil)\subseteq \varphi(a)$. So, $a\in\fil$ and $D_{\ult}^{\condi}(\fil)\subseteq \ultV$. By the monotonicity of $D_{\ult}^{\condi}$ we have that $D_{\ult}^{\condi}(\upop a)\subseteq \ultV$ and therefore $T_A(\ult,\varphi(a),\ultV)$. 

For the other direction, suppose that $T_A(\ult,\varphi(a),\ultV)$ for all $a\in\fil$. Then, $D_{\ult}^{\condi}(\upop a)\subseteq \ultV$ for all $a\in\fil$. It follows from Lemma~\ref{lem:props-for-Box} that $D_{\ult}^{\condi}(\fil)\subseteq \ultV$, thus $T_A(\ult,\varphi(\fil),\ultV)$.
\end{proof}

By Proposition \ref{prop:power-set-algebra} and \ref{T2}, we get that 
\begin{proposition}
For every conditional space $\frX\defeq\langle X,\tau,T\rangle$, the algebra 
\[
\Co(\frX)\defeq\langle\clopen(\tau),\Tcondi\rangle
\] 
is a conditional algebra.     
\end{proposition}

\pagebreak

From theorems~\ref{th:RT-for-CCA} and~\ref{th:Es-is-CS} we obtain
\begin{theorem}
    Every conditional algebra $\frA$ is isomorphic to the conditional algebra of a conditional space, i.e., to $\Co(\Es(\frA))$.
\end{theorem}

\begin{definition}
    Two conditional spaces $\frX_1\defeq\langle X_1,\tau_1,T_1\rangle$ and $\frX_2\defeq\langle X_2,\tau_2,T_2\rangle$ are isomorphic iff there exists a homeomorphism $\varepsilon\colon X_1\to X_2$ such that
    \[
        T_1(x,Z,y)\quad\text{iff}\quad T_2\left(\varepsilon(x),\varepsilon[Z],\varepsilon(y)\right)\,.
    \]\QED
\end{definition}

\begin{theorem}\label{th:CS-iso-Es}
    Every conditional space $\frX\defeq\langle X,\tau,T\rangle$ is isomorphic to the expanded Stone space $\Es(\Co(\frX))$ of the conditional algebra $\Co(\frX)=\langle\clopen(\tau),\Tcondi\rangle$ via the usual mapping $\varepsilon\colon X\to \Ul(\clopen(\tau))$, i.e.,   such that
 \[
 \varepsilon(x)\defeq\{U\in \clopen(\tau):x\in U\}\,.
 \]
\end{theorem}

\begin{proof}
We only need to show that
\[
T(x,Z,y)\quad\text{iff}\quad T_{\clopen(\tau)}(\varepsilon(x),\varepsilon[Z],\varepsilon(y))\,.
\]
Let $\upop Z=\{U\in\clopen(\tau):Z\subseteq U\}$. $\upop Z$ is a filter of the algebra $\langle \clopen(\tau),\Tcondi\rangle$, and if $Z$ is closed, $\varphi(\upop Z)=\varepsilon[Z]$.
\smallskip

 ($\Rarrow$)   Suppose $T(x,Z,y)$.  We will show that $D_{\varepsilon(x)}^{\Tcondi}(\upop Z)\subseteq \varepsilon(y)$. Let $U\in D_{\varepsilon(x)}^{\Tcondi}(\upop Z)$. Then, there exists $V\in \upop Z$ such that $V\Tcondi U\in \varepsilon(x)$, i.e., $Z\subseteq V$ and $x\in V\Tcondi U$. Since $T(x,Z,y)$, we get---by \eqref{df:Tcondi}---that $y\in U$ and therefore $U\in \varepsilon(y)$, as required. By \eqref{df:T_A}, $T_{\clopen(\tau)}(\varepsilon(x),\varepsilon[Z],\varepsilon(y))$.

\smallskip

 ($\Larrow$)   For the proof by contraposition, let $T^c(x,Z,y)$. We have two possibilities. In the first one, the relation fails due to $Z$ not being closed. Then,  since $\varepsilon$ is a homeomorphism, $\varepsilon[Z]$ is not closed either, and thus $T_{\clopen(\tau)}^c(\varepsilon(x),\varepsilon[Z],\varepsilon(y))$.
 
 In the second one, $Z$ is closed and by \ref{T3}, there exists $U\in\upop Z$ such that $T^c(x,U,y)$, i.e, $y\notin T(x,U)$. As $T(x,U)\in\closed(\tau)$, there exists $V\in \clopen(\tau)$ such that (a) $T(x,U)\subseteq V$ but (b) $V\notin \varepsilon(y)$. By (a) and Proposition~\ref{prop:monotonicity-for-T} we obtain that $U\Tcondi V\in\varepsilon(x)$ and thus $V\in D_{\varepsilon(x)}^{\Tcondi}(\upop U)$. The more so---by Lemma~\ref{lem:D-filter}---$V\in D_{\varepsilon(x)}^{\Tcondi}(\upop Z)$. So $D_{\varepsilon(x)}^{\Tcondi}(\upop Z)\nsubseteq\varepsilon(y)$, i.e., $T_{\clopen(\tau)}^c(\varepsilon(x),\varphi(\upop Z),\varepsilon(y))$, which ends the proof.
\end{proof}

\section{Categorical duality for conditional algebras}\label{sec:categorical-duality}

As is well known, Boolean homomorphisms correspond in a one-to-one manner to continuous functions of the dual spaces. If $h\colon A\to B$ is a homomorphism of Boolean algebras, then the function $f_h\colon \Ul(B)\to\Ul(A)$ given by
\[
f_h(\ult)\defeq h^{-1}[\ult]
\] 
is a continuous function. On the other hand, if $f\colon X_1\to X_2$ is continuous function between Boolean spaces $\langle X_1,\tau_1\rangle$ and $\langle X_2,\tau_2\rangle$, then the map $h_f\colon \clopen(\tau_2)\to \clopen(\tau_1)$ such that 
\[
h_f(U)\defeq f^{-1}[U]
\]
is a homomorphism of the Boolean algebras. We are going to make use of this correspondence to extend it to maps between conditional algebras and conditional spaces.

\begin{definition}
    A homomorphism $h\colon A \to B$ of conditional algebras is a homomorphism of Boolean algebras that also satisfies 
    \begin{equation}\label{homomor}
      h(a\condi b)=h(a)\condi h(b).
     \end{equation}
     We will call such an $h$ a \emph{conditional} homomorphism.\QED
\end{definition}

\begin{lemma}\label{lem:c-homo-via-T}
    Let $\frA\defeq\klam{A,\condi},\frB\defeq\klam{B,\condi}\in\CA$ and $h\colon A \to B$ be a Boolean homomorphism. $h$ is a conditional homomorphism if and only if for all $\ult\in \Ul(B)$, $a,b\in A$ the following equivalence holds
    \[
    T_B(\ult,f_h^{-1}[\varphi(a)])\subseteq f_h^{-1}[\varphi(b)]\quad\text{iff}\quad T_A(f_h(\ult),\varphi(a))\subseteq \varphi(b).
   \]
\end{lemma}
\begin{proof}

($\Rarrow$)  Assume that $h$ is a homomorphism of conditional algebras. For the proof by contraposition, assume that $\ultV\in \Ul(A)$ is an element of $T_A(f_h(\ult),\varphi(a))$ but $b\notin \ultV$. Then, $\D{f_h(\ult)}{\upop a}\subseteq \ultV$, and $b\notin\D{f_h(\ult)}{\upop a}$, i.e., $a\condi b\notin h^{-1}[\ult]$. So, $h(a)\condi h(b)=h(a\condi b)\notin \ult$, and applying \eqref{eq:cond-order-reversing} we obtain that $h(b)\notin\D{\ult}{\upop h(a)}$. It follows that there exists $\ultW\in \Ul(B)$ such that
\[
\ultW\in T_B(\ult,\varphi(h(a)))=T_B(\ult,f_h^{-1}[\varphi(a)]) 
\]
but $\ultW\notin \varphi (h(b))=f_h^{-1}[\varphi(b)]$, as required.

Analogously, let $\ultV$ be an ultrafilter of $B$ such that $\ultV\in T_B(\ult,f_h^{-1}[\varphi(a)])\setminus f_h^{-1}[\varphi(b)]$. Thus $\D{\ult}{\upop h(a)}\subseteq \ultV$ and $h(b)\notin\ultV$, since $f_h^{-1}[\varphi(b)]=\varphi(h(b))$. Then, $h(a\condi b)=h(a)\condi h(b)\notin\ult$ and we get $a\condi b\notin h^{-1}[\ult]=f_h(\ult)$. From \eqref{eq:cond-order-reversing} it follows again that $b\notin\D{f_h(\ult)}{\upop a}$, so there is an ultrafilter $\ultW\notin\varphi(b)$ that extends $\D{f_h(\ult)}{\upop a}$. In consequence $T_A(f_h(\ult),\varphi(a))\nsubseteq \varphi(b)$.

    \smallskip
    
($\Larrow$) Suppose that the equivalence holds. Let $a,b\in A$ and $\ult\in \Ul(B)$ be such that $h(a \condi b)\in \ult$. Then, $a\condi b\in h^{-1}[\ult]=f_h(\ult)$ and we get $b\in D_{f_h(\ult)}(\upop a)$. Thus, $T_A(f_h(\ult),\varphi(a))\subseteq \varphi(b)$. By assumption, $  T_B(\ult,f_h^{-1}[\varphi(a)])\subseteq f_h^{-1}[\varphi(b)]$ and it follows that $h(b)\in D_\ult(\upop h(a))$, i.e., $h(a)\condi h(b)\in \ult$. Therefore, we get that $h(a\condi b)\leq h(a)\condi h(b)$. The other inequality is proven in an analogous way. 
\end{proof}

Making use of Lemma \ref{lem:c-homo-via-T} we introduce the following
\begin{definition}
Let $\frX_1\defeq\langle X_1,\tau_1,T_1\rangle$ and $\frX_2\defeq\langle X_2,\tau_2,T_2\rangle$ be conditional spaces. A continuous function $f\colon X_1\to X_2$ between Boolean spaces $\langle X_1,\tau_1\rangle$ and $\langle X_2,\tau_2\rangle$ is a \emph{conditional} function if for all $x\in X_1$ and $U,V\in \clopen(\tau_2)$
\begin{align*}\label{CF}\tag{$\mathrm{CF}$}
 T_1(x,f^{-1}[U])\subseteq f^{-1}[V]\quad\text{iff}\quad T_2(f(x),U)\subseteq V.
\end{align*}\QED
\end{definition}
Observe that the identity function trivially satisfies \eqref{CF}.

\begin{lemma}\label{conhom}
    Condition \eqref{CF} is equivalent to
\[
h_f(U\condi_{T_2}V)=h_f(U)\condi_{T_1}h_f(V)
\]
for all $U,V \in \clopen(\tau_2)$, i.e., the function $h_f\colon \clopen(\tau_2) \to \clopen(\tau_1)$ is a conditional homomorphism. 
\end{lemma}
\begin{proof}
    ($\Rarrow$) Suppose that $f$ satisfies \eqref{CF}. Then we have $x\in h_f(U\condi_{T_2}V)$ iff $f(x)\in U\condi_{T_2}V$ iff $T_2(f(x),U)\subseteq V$ iff $T_1(x,f^{-1}[U])\subseteq f^{-1}[V]$ iff $x\in h_f(U)\condi_{T_1}h_f(V)$. Thus, the equality follows.

    \smallskip

    ($\Larrow$) Suppose that $h_f(U\condi_{T_2}V)=h_f(U)\condi_{T_1}h_f(V)$. Then, $T_1(x,f^{-1}[U])\subseteq f^{-1}[V]$ iff $x\in h_f(U)\condi_{T_1}h_f(V)$ iff $x\in h_f(U\condi_{T_2}V)$  iff $x\in f^{-1}[U\condi_{T_2}V]$ iff $T_2(f(x),U)\subseteq V$.
\end{proof}

Conditional algebras with conditional homomorphisms form a category in which the identity arrow is the identity homomorphism and the composition is the usual composition of functions. Also, conditional spaces with conditional functions form a category in which the identity arrow is given by the identity function and the composition is the usual composition of functions. We will denote these categories by $\mathbf{CA}$ and $\mathbf{CS}$, respectively.

Notice that Theorem \ref{th:Es-is-CS} and Lemma \ref{lem:c-homo-via-T} allow to define a contravariant functor $G\colon \mathbf{CA} \to \mathbf{CS}$ as follows:
\begin{align*}
    \langle A,\condi\rangle & \mapsto  \langle \Ul(A),\topos,T_A\rangle\,,\\
    h & \mapsto  f_h.
\end{align*}
Thanks to Proposition \ref{prop:power-set-algebra}, condition \ref{T2} and Lemma \ref{conhom} we can define a contravariant functor $H\colon \mathbf{CS} \to \mathbf{CA} $ in the following way:
\begin{align*}
    \langle X,\tau, T\rangle & \mapsto  \langle \clopen(\tau),\Tcondi\rangle\,,\\
    f & \mapsto  h_f.
\end{align*}

\begin{theorem}
    The categories $\mathbf{CA}$ and $\mathbf{CS}$ are dually equivalent.
\end{theorem}

\begin{proof}
    Let $\langle A,\condi\rangle\in \CA$. By Theorem \ref{th:RT-for-CCA} the map $\varphi\colon A\to \clopen(\topos)$ defines a natural isomorphism between $H \circ G$ and $\mathrm{Id}_{\mathbf{CA}}$. On the other hand, if $\langle X,\tau,T\rangle$ is a conditional space, then from Theorem \ref{th:CS-iso-Es} we get that the map $\varepsilon\colon X \to \Ul(\clopen(\tau))$ leads to a natural isomorphism between $G \circ H$ and $\mathrm{Id}_{\mathbf{CS}}$. This concludes the proof.
\end{proof}

\section{Characterization of conditional subalgebras}

In this section, we present a characterization of the subalgebras of
conditional algebras.

\begin{definition}
 Let $E$ be an equivalence relation on a Boolean space
$X$, end let $E(x)\defeq\{y\in X\mid E(x,y)\}$. A subset $U$ of $X$ is \emph{closed under} $E$ if
$x\in U$ implies $E(x)\subseteq U$. An equivalence relation $E$ on $X$ is said to be a \emph{Boolean
equivalence} if for each pair $\klam{x,y}\notin E$ there
exists a~clopen subset $U$ of $X$ that is closed under $E$ and such that $x\in U$
and $y\notin U$. 
\end{definition}

For $U\subseteq X$ let
\[
E(U)\defeq\{y\in X\mid(\exists x\in U)\,E(x,y)\}=\bigcup_{x\in U}E(x)\,.
\]

If $A$ is a Boolean algebra and $B$ is a subalgebra of $A$,
then it is easy to show that the relation
\[
E_{B}\defeq\left\{\klam{u,v}\in\Ul(A)\times\Ul(A):u\cap B=v\cap B\right\} 
\]
 is a Boolean equivalence. Conversely, if $E$ is a Boolean equivalence
of a Boolean space $X$ then 
\[
B_{E}\defeq\left\{U\in\clopen(\tau)\colon E(U)=U\right\} 
\]
 is the domain of a subalgebra of the Boolean algebra $\clopen(\tau)$.
Moreover, the correspondence $B\mapsto E_{B}$ is a dual isomorphism
between the lattice of domains of the subalgebras of $A$ and the
lattice of the Boolean equivalence of the Stone space of $A$ (see \citealp{Koppelberg-TD}).

Our next aim is to prove that there is a bijective correspondence
between subalgebras of conditional algebras and certain Boolean
equivalences with an additional condition.

\begin{definition}
  Let $E$ be a Boolean equivalence on a space $X$. We define the binary relation $\preceq_{E}$ on $\closed(\tau)$
as follows: $Y\preceq_{E} C$ if and only if for every $y\in Y$ there exists $x\in X$ such that $x\in C$
and $E(x,y)$.   
\end{definition}

\begin{lemma} \label{lemma_B} Let $A$ be a Boolean algebra and
let $B$ be a subalgebra of $A$. Let $\fil,\filH\in\Fil(A)$. Then
\[\varphi(\fil)\preceq_{E_B}\varphi(\filH)\qquad\text{iff}\qquad \filH\cap B\subseteq\fil\,.\]
\end{lemma}
\begin{proof}
($\Rarrow$) Suppose that $\varphi(\fil)\preceq_{E_B}\varphi(\filH)$. Let $a\in \filH\cap B$ and $a\notin \fil$.
Then there is $u\in\Ul(A)$ such that $a\notin u$ and $\fil\subseteq u$.
Since $\varphi(\fil)\preceq_{E_B}\varphi(\filH)$, there exists $v\in\Ul(A)$ such that $\filH\subseteq v$
and $v\cap B=u\cap B$. Thus, $a\in \filH\cap B\subseteq v\cap B=u\cap B\subseteq u$,
so $a\in u$, which is a contradiction. Hence, $\filH\cap B\subseteq \fil$.

\smallskip

($\Larrow$) Conversely, suppose that $\filH\cap B\subseteq\fil$. Let $u\in\Ul(A)$ be such that $\fil\subseteq u$.
Then $\fil\cap B\subseteq u$, so it follows from hypothesis that $\filH\cap B\subseteq u$.
Observe that the set $\filH\cup\left(u\cap B\right)$ has the finite intersection property. Indeed, if there are $a\in \filH$
and $b\in u\cap B$ such that $a\wedge b=0$, then $a\leq\neg b$ and
$\neg b\in \filH$. Since $B$ is a Boolean subalgebra and $b\in B$,
we get that $\neg b\in \filH\cap B\subseteq \fil\subseteq u$, which is a contradiction.
In consequence, there exists an ultrafilter $v$ such that
$\filH\subseteq v$ and $u\cap B\subseteq v$. As $u\cap B\subseteq v\cap B$,
and $u\cap B$, $v\cap B\in\Ul(B)$, we obtain that $u\cap B=v\cap B$, i.e., $E_B(u,v)$. Thus,
$\varphi(\fil)\preceq_{E_B}\varphi(\filH)$, as required.
\end{proof}
\begin{definition}
     Let $\frX=\langle X,\tau, T\rangle$ be a conditional space and let $E$ be a Boolean equivalence on the space $X$. We say that $E$ is a \emph{C-equivalence} if it satisfies the following condition:
     for each $x,x',y\in X$ and each $Y\in \closed(\tau)$, if $E(x,y)$
and $T(x,Y,x')$, then there exist $y'\in X$ such that $E(x',y')$ and
$C\in \closed(\tau)$ such that $T(y,C,y')$ and $C\preceq_{E}Y$.
\end{definition}

\begin{theorem}\label{th:chara_subalgebra}
Let $\frA\defeq\langle A,\condi\rangle,\frB\defeq\langle B,\condi\rangle\in\CA$. Suppose
that $\frB$ is a Boolean subalgebra of $\frA$. Then, the following conditions
are equivalent\/\textup{:}
\begin{enumerate}
\item $E_B$ is a C-equivalence.
\item $\frB$ is a subalgebra of $\frA$.
\end{enumerate}
\end{theorem}
\begin{proof}
$(1)\Rarrow(2)$ Let (1) hold and assume towards a contradiction that there are $a,b\in B$ such that $a\condi b\notin B$.
Consider the filter $\fil$ generated by $\upop (a\condi b)\cap B$ and the principal ideal $\ide\defeq\downop(a\condi b)$. Suppose there is $c\in\ide\cap\fil$. Thus, $c\leq a\condi b$ and there is a $d\in\upop(a\condi b)\cap B$ such that $d\leq c$.
Then $d=a\condi b$, and in consequence $a\condi b\in B$, which is a contradiction. Thus,
$\ide\cap \fil=\emptyset$. Let $u\in\Ul(A)$ be such that $\ide\cap u=\emptyset$
and $\fil\subseteq u$. So $a\condi b\notin u$ and $\upop(a\condi b)\cap B\subseteq u$.
The former condition together with Lemma~\ref{lem:existence-for-representation} entail existence of $v\in\Ul(A)$ and $\filG\in\Fi(A)$ such that ($\dagger$) $T_A(u,\varphi(\filG),v)$,
$a\in \filG$ and $b\notin v$. On the other hand, $B\cap u^c$ is closed under finite joins, so the ideal (in $A$) generated by the intersection is just $\downop (B\cap u^c)$. Therefore, the latter condition entails that $a\condi b$ is not an element of the ideal.
Then, there exists $w\in\Ul(A)$ such that $B\cap u^{c}\cap w=\emptyset$
and $a\condi b\in w$. Thus, ($\ddagger$) $u\cap B=w\cap B$.
By ($\dagger$) and ($\ddagger$) there exist $z\in\Ul(A)$
and $\filH\in\Fil(A)$ such that $T_A(w,\varphi(\filH),z)$ and
$\filG\cap B\subseteq \filH$. As $a\in \filG\cap B$, we get $a\in \filH$, and as
$a\condi b\in w$, we have $b\in z$. So, $b\in z\cap B=v\cap B$,
i.e., $b\in v$, which is a contradiction. Thus, $a\condi b\in B$,
and consequently $\langle B,\condi\rangle$ is a subalgebra of $\langle A,\condi\rangle$.

\smallskip

$(2)\Rarrow(1)$ Suppose $\langle B,\condi\rangle$ is a subalgebra of $\langle A,\condi\rangle$. Let $u,v$ and $w$ be ultrafilters of $A$. Take an arbitrary $\fil\in\Fil(A)$, and assume that
$u\cap B=w\cap B$ and $T_{A}(u,\varphi(\fil),v)$, i.e., $\D{u}{\fil}\subseteq v$.
Since $\fil\cap B$ is closed under finite meets, $\filG\defeq\upop(\fil\cap B)$ is a filter. We will prove that $\D{w}{\filG}\cap\downop(v^{c}\cap B)$ is empty. Suppose otherwise, and take an element $c$ from the intersection.
Then there exist $a$ and $d$ such that: $a\in \fil\cap B$ (by \ref{eq:cond-order-reversing}), $d\notin v$ and $d\in B$, $a\condi c\in w$ and $c\leq d$. By \eqref{eq:cond-order-preserving}, $a\condi d\in w$.
As $a,d\in B$, $a\condi d\in B\cap w=u\cap B$, i.e.,
$a\condi d\in u$. Since $a\in \fil$ and $\D{u}{\fil}\subseteq v$,
we obtain that $d\in v$, which is a contradiction. Thus, $\D{w}{\filG}\cap\downop(v^{c}\cap B)=\emptyset$.
Then there exists $z\in\Ul(A)$ such that $\D{w}{\filG}\subseteq z$ and
$v\cap B=z\cap B$. By its construction, we have that $\fil\cap B\subseteq\filG$.
In this way we have proven that there are $z\in\Ul(A)$ and $\filG\in\Fil(A)$
such that $T_A(w,\varphi(\filG),z)$ and $\varphi(\filG)\preceq_{E_B}\varphi(\fil)$, as required.
\end{proof} 
\begin{theorem} Let $\frA\in\CA.$ The correspondence
$B\mapsto E_{B}$ is a dual isomorphism between the lattice of domains
of subalgebras of $\frA$ and the lattice of the
C-equivalences of the expanded Stone space of $\frA$.
\end{theorem}

\section{Congruences of conditional algebras}

This section aims to study the congruences of conditional algebras. We use the categorical duality
developed in Section~\ref{sec:categorical-duality} in order to prove that for every conditional
algebra $\frA$ there is a dual isomorphism between
the congruences of $A$ and a family of closed subsets of
the dual space of $\frA$. 

Let $\Cong(A)$ be the lattice of congruences of an algebra $A$. For every $a\in A$ and $\theta\in\Cong(A)$, $a/\theta$ is an equivalence class of $a$ with respect to $\theta$.  $A/\theta$ is the quotient algebra of $A$. According to a well-known folklore result on  the Stone duality,  for every closed subset $Y$ of $\Ul(A)$, the set
\[
\theta(Y)\defeq\{\klam{a,b}\in A\times A\colon Y\cap\varphi(a)=Y\cap\varphi(b)\} 
\]
is a congruence on $A$. Moreover, the assignment $Y\mapsto\theta(Y)$
establishes a dual isomorphism between the lattice of the closed subsets of the Stone space $\Ul(A)$ and the lattice $\Cong(A)$.

Let $\langle X,\tau,T\rangle$ be a conditional space. For any pair of elements $x$ and $y$ we define
\[
\mathbf{C}(x,y)\defeq\{Z\in\closed(\tau)\mid T(x,Z,y)\}.
\]
We will see that it is closed under intersection of chains. Let us consider a chain $\mathcal{C}$ of members of $\mathbf{C}(x,y)$. Suppose $\bigcap \mathcal{C}\notin \mathbf{C}(x,y)$. Then, there exists $U\in \clopen(\tau)$ such that $T^c(x,U,y)$ and $\bigcap \mathcal{C}\subseteq U$. By compactness of $U$, we get that there exists $C\in \mathcal{C}$ such that $C\subseteq U$. But by \ref{T3}, it follows $T(x,U,y)$, a contradiction.

So, from the Kuratowski-Zorn lemma, we get that the set $\{X \setminus C : C \in \mathcal{C}\}$ has maximal elements (with respect to set inclusion), provided it is non-empty. Dually, $\mathcal{C}$ has minimal elements, provided it is non-empty. Therefore, we can introduce the following definition.

\begin{definition}
    Let $\langle X,\tau,T\rangle$ be a conditional space and let $Y\in\closed(\tau)$.  We say that $Y$ is a \emph{$T$-closed set} if
for all $x,y\in X$, if $x\in Y$ and $Z$ is minimal in $\mathbf{C}(x,y)$, then $Z\subseteq Y$ and $y\in Y$.\QED  
\end{definition}

In the proof of the next theorem, we will use the following standard result.
\begin{lemma}
    If $Y$ is a closed subset of the Stone space of $A$ and $\klam{a,b}\in\theta(Y)$, then there exists an $x\in\fil_Y$ such that $a\wedge x=b\wedge x$.
\end{lemma}
\begin{proof}
If $Y$ is closed, then $Y\cap\varphi(a)$ is closed too, so there is a filter $\fil_1$ such that  $Y\cap\varphi(a)=\varphi(\fil_1)$. Consider the filter $\fil_Y^a$ generated by $\fil_Y\cup\{a\}$. We aim to show that $\fil_1=\fil_Y^a$.

\smallskip

We have that $\varphi(\fil_1)\subseteq Y$ and $\varphi(\fil_1)\subseteq \varphi(a)$. Thus $\fil_Y\subseteq\fil_1$ and $a\in\fil_1$. So $\fil_Y\cup\{a\}\subseteq\fil_1$ and $\fil_Y^a\subseteq\fil_1$. For the other direction, assume that $\fil_1\nsubseteq \fil_Y^a$. Therefore there exists an ultrafilter $\ult$ such that $\fil_Y^a\subseteq\ult$ and $\fil_1\nsubseteq\ult$. But $\ult\in\varphi(a)\cap Y=\varphi(\fil_1)$ and thus $\ult\in\varphi(\fil_1)$, a contradiction.

In an analogous way we demonstrate that  the filter $\fil_2$ such that $Y\cap\varphi(b)=\varphi(\fil_2)$ is equal to $\fil_Y^b$, the filter generated by $\fil_Y\cup\{b\}$.

\smallskip

Thus $\fil_Y^a=\fil_Y^b$. Since $b\in\fil_Y^a$, there is $y_1\in\fil_Y$ such that $y_1\wedge a\leq b$. Analogously, there is $y_2\in\fil_Y$ such that $y_1\wedge b\leq a$. In consequence:
\[a\wedge y_1\wedge y_2=b\wedge y_1\wedge y_2\]
and $y_1\wedge y_2\in\fil_Y$.
\end{proof}

\begin{theorem} \label{cong and M-closed} Let $\frA\defeq\langle A,\condi\rangle\in\CA.$
Let $Y$ be a closed subset of the Stone space of $A$. Then
the following conditions are equivalent\/\textup{:}
\begin{enumerate}
\item $\theta(Y)$ is a congruence on $\frA$, 
\item $Y$ is a $T$-closed set.
\end{enumerate}
\end{theorem}
\begin{proof}
Fix a closed subset $Y$ of $\Ul(A)$.

\smallskip

$(1)\Rarrow (2)$ Suppose that $(1)$ is satisfied. 

To prove that $Y$ is a $T$-closed set, assume that $u,v\in\Ul(A)$ are such such that $\fil_Y\subseteq u$, and let $\filH\in \Fi(A)$ such that $\varphi(\filH)$ is a minimal element in $\mathbf{C}(u,v)$.
Towards a contradiction, consider a scenario with an $a\in \fil_Y$ which is not an element of $\filH$, and take the filter $\filH_{a}$ generated by $\filH\cup\{a\}$. Since $\varphi(\filH_a)\subseteq \varphi(\filH)$, the minimality of $\filH$ entails that 
$T_{A}^c(u,\varphi(\filH_{a}),v)$, i.e, $\D{u}{\filH_{a}}\nsubseteq v$.
Pick $b\in\D{u}{\filH_a}$ such that $b\notin v$. By the former, there is $x\in \filH$ such that $(x\wedge a)\rightarrow b\in u$. By the definition of $\fil_Y$, we have that $Y\subseteq\varphi(a)$, thus $\klam{a,1}\in\theta(Y)$, and we get that $\klam{x\wedge a,x}\in\theta(Y)$. So, $\klam{(x\wedge a)\condi b,x\condi b}\in\theta(Y)$.
Thus, $\varphi((x\wedge a)\condi b)\cap Y=\varphi(x\condi b)\cap Y$.
As $\fil_Y\subseteq u$ and $(x\wedge a)\condi b\in u$, we
have that $x\condi b\in u$. Since $T_{A}(u,\varphi(\filH),v)$
and $x\in \filH$, we get $b\in v$, which is a contradiction. Thus, $\fil_Y\subseteq \filH$.

We prove now that $\fil_Y\subseteq v$. Suppose otherwise and pick an $a\in \fil_Y$ such that $a\notin v$. Consider the
ideal $\ide$ generated by $v^{c}\cup\{\neg a\}$.
If $\D{u}{\filH}\cap\ide=\emptyset$, then there exists $z\in\Ul(A)$ such that $\D{u}{\filH}\subseteq z$
and $v=z$ and $a\in z$, which is a contradiction. Thus, $\D{u}{\filH}\cap\ide\neq\emptyset$ and there are
$b\in\D{u}{\filH}\cap\ide$ and
$c\notin v$ such that $b\leq c\vee\neg a$. So, there is $x\in \filH$
such that $x\condi b\in u$, and consequently $x\condi b\leq x\condi(c\vee\neg a)\in u$.
As $\klam{a,1}\in\theta(Y)$, it follows that $\klam{\neg a,0}\in\theta(Y)$ and $\klam{c\vee\neg a,c}\in\theta(Y)$, and so $\klam{x\condi(c\vee\neg a),x\condi c}\in\theta(Y)$.
Since $\fil_Y\subseteq u$ and $\varphi((x\condi(c\vee\neg a))\cap Y=\varphi(x\condi c)\cap Y$,
we have $x\condi c\in u$. As $\D{u}{\filH}\subseteq v$  and $x\in\filH$, we get $c\in v$, which is a contradiction. Thus, $\fil_Y\subseteq v$.

\smallskip

$(2)\Rarrow (1)$ Suppose that $Y$ is a $T$-closed set. We will show that $\theta(Y)$ is a
congruence on $\klam{A,\condi}$. Let $a,b,c,d\in A$ be such that $\klam{a,b},\klam{c,d}\in\theta(Y)$. Suppose that $\klam{a\condi c,b\condi d}\notin\theta(Y)$,
i.e., $\varphi(a\condi c)\cap Y\neq\varphi(b\condi d)\cap Y$. Without the loss of generality, assume there exists an ultrafilter $u$ such that $a\condi c\in u$,
$\fil_Y\subseteq u$ and $b\rightarrowtriangle d\notin u$. By Lemma~\ref{lem:existence-for-representation}, there are a filter $\filG$ and an ultrafilter $v$ such that $\D{u}{\filG}\subseteq v$, $b\in \filG$ and $d\notin v$. Since $ \mathbf{C}(u,v)\neq \emptyset$, there exists a filter $\filH$ such $\filG\subseteq \filH$ and $\varphi(\filH)$ is minimal in $\mathbf{C}(u,v)$.
Then $b\in \filH$. As $\klam{a,b}\in\theta(Y)$, there exists $y\in \fil_Y$ such
that $a\wedge y=b\wedge y$. But, since $Y$ is a  $T$-closed set, $\fil_Y\subseteq \filH$, so we have that $y\in \filH$.
Thus, $b\wedge y=a\wedge y\in \filH$. Then $a\in \filH$. Since $T_{A}(u,\varphi(\filH),v)$,
$a\rightarrowtriangle c\in u$, and $a\in \filH$, we get $c\in v$. Then, since $v\in Y$ and by assumption we have that $\varphi(c)\cap Y=\varphi(d)\cap Y$, we get that $d\in v$, a contradiction.
\end{proof}
The following result is a consequence of Theorem \ref{cong and M-closed}.

\begin{theorem} Let $\frA\defeq\langle A,\condi\rangle\in\CA.$ There exists
a dual isomorphism between the lattice ${\rm {Con}}(\frA)$ of congruences
of $\frA$ and the lattice of $T$-closed sets of $\Es(\frA)$.
\end{theorem}

\section{Characterizations of some subvarieties of \texorpdfstring{$\CA$}{CA}}\label{sec:subvarieties}

So far, we have studied what can be seen as a very general algebraic interpretation of conditionals. Now, we aim to show that the techniques and results from the previous sections can be applied to various classes of algebras, for which we will provide a
relational characterization by means of the $T$ relation.

To begin, let us list the conditions we are going to consider in this section:
\begin{gather}
0\condi a=1\tag{C1$^\ast$}\label{C1-for-0}\\
(a\condi c)\wedge(b\condi c)\leq(a\vee b)\condi c\,,\tag{C3$^\ast$}\label{C3-reversed}\\
\tag{C4}\label{C4} a\condi b\leq c\condi(a\condi b)\,,\\
\tag{C5}\label{C5} a\wedge(a\condi b)\leq b\,,\\
\tag{C6}\label{C6} a\condi b\leq\neg b\condi\neg a\,,\\
\tag{C7}\label{C7} \neg(a\condi b)\leq c\condi\neg(a\condi b)\,,\\
\tag{C8}\label{C8} (1\condi (\neg a\lor b))\land (b\condi c)\leq a\condi c\,.
\end{gather}

Let us return to the subvarieties of $\CA$ mentioned in the introduction: 
\begin{enumerate}
\item pseudo-subordination algebras 
    \[
        \PSB\defeq\CA+\{\eqref{C1-for-0},\eqref{C3-reversed}\}\,,
    \]

    \item pseudo-contact algebras
    \begin{align*}
    \PsC\defeq{}&\PSB+\{ \eqref{C5},\eqref{C6}\}\\
    ={}&\CA+\{\eqref{C1-for-0},\eqref{C3-reversed},\eqref{C5},\eqref{C6}\}\,,
    \end{align*}
    \item strict-implication algebras
    \begin{align*}
    \SIA\defeq{}&\PSB+\{ \eqref{C4}, \eqref{C5},\eqref{C7}, \eqref{C8} \}\\
    ={}&\CA+\{\eqref{C1-for-0},\eqref{C3-reversed},\eqref{C4},\eqref{C5},\eqref{C7},\eqref{C8}\}\,,
    \end{align*}
    \item symmetric strict-implication algebras
    \begin{align*}
        \StwoIA\defeq{}&\PSB+\{ \eqref{C4}, \eqref{C5},\eqref{C6}, \eqref{C7}\}\\
        ={}&\CA+\{\eqref{C1-for-0},\eqref{C3-reversed},\eqref{C4},\eqref{C5},\eqref{C6},\eqref{C7}\}\,.
    \end{align*}
\end{enumerate}

\begin{figure}\centering
\begin{tikzpicture}
\node (n1) at (0,1) {$\CA$};
\node (n2) at (0,0) {$\PSB$};
\node (n3) at (-1.5,-1) {$\SIA$};
\node (n5) at (1.5,-1) {$\PsC$};
\node (n6) at (0,-2) {$\StwoIA$};

\draw  (n1) -- (n2);
\draw  (n2) -- (n3);
\draw  (n2) -- (n5);
\draw  (n3) -- (n6);
\draw  (n5) -- (n6);
\end{tikzpicture}\caption{The poset of subvarieties of the variety of conditional algebras analyzed in Section \ref{sec:subvarieties}.}
\end{figure}
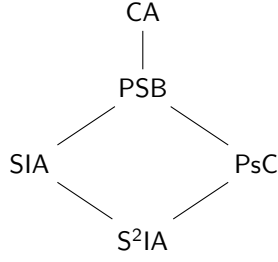

For the varieties $\SIA$ and $\StwoIA$ we use the equational axiomatizations introduced in \citep{Celani-et-al-AVOACRTSO}. In the paper, it has been proved that in the variety of pseudo-subordination algebras, $\eqref{C6}$ entails $\eqref{C8}$. In light of the fact that the reverse dependency does not hold, the former condition marks the difference between the two varieties.

Below, we prove that: 
\begin{enumerate}
    \item we may characterize the variety of pseudo-subordination algebras via properties of its dual spaces,
    \item each of the conditions \eqref{C4}--\eqref{C8} listed above has its first- or second-order correspondent expressing a~property of the $T_A$ relation,
    \item all the four subvarieties of $\CA$ are $\pi$-canonical. 
\end{enumerate}


\subsection{Dual spaces of pseudo-subordination algebras}

For a Boolean algebra $A$ let
\[
\closed_+(\tau_s)\defeq\closed(\topos)\setminus\{\emptyset\}\,.
\]
By definition, $T_A\subseteq\Ul(A)\times\closed(\topos)\times\Ul(A)$. For the condition \eqref{C1-for-0} we have the following
\begin{lemma}\label{lem:T-middle-non-empty}
 Let $\frA\defeq\klam{A,\condi}$ be a conditional algebra. Then, \eqref{C1-for-0} holds in $\frA$ iff $T_A\subseteq \Ul(A)\times \closed_+(\tau_s)\times \Ul(A)$.
\end{lemma}
\begin{proof}
    $(\Rarrow)$ Suppose that there exists $a\in A$ such that $0\condi a\neq 1$. So, there exists $\ult\in \Ul(A)$ such that $0\condi a\notin \ult$. In consequence, $a\notin D^\to_\ult(A)$, which means that $D^\to_\ult(A)$ is a proper filter and as such can be extended to $\ultV\in \Ul(A)$: $D^\to_\ult(A)\subseteq\ultV$. As $\varphi(A)=\emptyset$, by definition, we get that $ T_A(\ult, \emptyset,\ultV)$.

    \smallskip

    $(\Larrow)$ Let $\ult\in\Ul(A)$. If $0\condi a=1$ for all $a\in A$, then $D^\to_\ult(A)=A$ and therefore $T_A(\ult,\emptyset)=\varphi(D^\to_\ult(A))=\emptyset$.
\end{proof}

\begin{lemma}\label{lem:fundamental} Let
$\frA\defeq\klam{A,\condi}\in\CA+\eqref{C3-reversed}$. Let $\fil,\filH\in\Fi(A)$, $\fil\neq A$ and $\ult\in\Ul(A)$.
If $D_{\filH}^{\condi}(\fil)\subseteq \ult$, then there exists $\ultV\in\varphi(\fil)$
such that $D_{\filH}^{\condi}(\ultV)\subseteq \ult$.
\end{lemma}
\begin{proof}
Let $D_{\filH}^{\condi}(\fil)\subseteq u$. Consider the family 
\[
\bfF\defeq\{ \filG\in\Fi(A):D_\filH^\condi(\filG)\subseteq \ult\text{ and }\fil\subseteq \filG\}\,. 
\]
The family is non-empty, as $\fil\in\bfF$, and---as it is routine to verify---for any non-empty chain $C\subseteq\bfF$, $\bigcup C\in\bfF$. Thus, by Kuratowski-Zorn Lemma $\bfF$ has a maximal element $\fil^\ast$. 

 If $\fil^\ast=A$ then we can consider an arbitrary ultrafilter $\ultV$ extending $\fil$. By monotonicity of $D_{\filH}^{\condi}$ we have $D^\condi_\filH(\ultV)\subseteq D^\condi_\filH(A)\subseteq\ult $.

Let then $\fil^\ast\subsetneq A$. We are going to show that $\fil^\ast$ is prime, and so an ultrafilter. To this end, we assume that $a\vee b\in\fil^\ast$. For the sake of contradiction, let $a\notin\fil^\ast$ and $b\notin\fil^\ast$. Let us consider the filters $\fil^\ast_{a}$ and $\fil^\ast_{b}$ generated by $\fil^\ast\cup\{ a\}$ and $\fil^\ast\cup\{b\}$, respectively. Since $\fil^\ast$ is a proper subset of both and is maximal in $\bfF$, neither of the two filters is in $\bfF$, which means that
\[
\D{\filH}{\fil^\ast_{a}}\nsubseteq \ult\quad\text{and}\quad \D{\filH}{\fil^\ast_{b}}\nsubseteq \ult\,.
\]
In consequence there exist $z_1\in\fil^\ast_{a}$, $z_2\in\fil^\ast_{b}$ and $x_1,x_2\notin \ult$ such
that
\[
z_1\condi x_{1},z_2\condi x_{2}\in \filH.
\]
By construction, there are $y_1,y_2\in \fil^\ast$ such that $y_1\wedge a\leq z_1$
and $y_2\wedge b\leq z_2$. Put $y\defeq y_1\wedge y_2$ and $x\defeq x_1\vee x_2\notin \ult$. By \eqref{eq:cond-order-reversing} both $y\wedge a\condi x$ and $y\wedge b\condi x$ are in $\filH$, and so by \eqref{C3-reversed}
\[
(y\wedge a\condi x)\wedge(y\wedge b\condi x)\leq y\wedge(a\vee b)\condi x\in \filH.
\]
Since $D^\condi_\filH(\fil^\ast)\subseteq \ult$ and $y\wedge(a\vee b)\in\fil^\ast$ it follows that $x\in\ult$, a~contradiction.
\end{proof}
\begin{lemma}\label{cond C3*}
    Let $\frA\defeq\klam{A,\condi}\in\CA$. If the relation $T_A$ of its dual space $\Es(\frA)$ satisfies
    \[(\forall \ult,\ultV\in \Ul(A))(\forall Y\neq \emptyset)\,(T_A(\ult, Y,\ultV)\Rarrow(\exists \ultW\in Y)\,T_A(\ult,\{\ultW\},\ultV)\] then $\frA$ satisfies \eqref{C3-reversed}.
\end{lemma}

\begin{proof}
Suppose that there exist $a,b,c\in A$ such that $(a\condi c)\wedge (b\condi c)\nleq (a\vee b\condi c)$. Then, there exists an ultrafilter $\ult$ such that $(a\condi c)\wedge (b\condi c)\in \ult$ but $(a\vee b\condi c)\notin \ult$. Thus there exists $\ultV\in \Ul(A)$ such that $T(\ult,\varphi(a\vee b),\ultV)$ but $c\notin \ultV$. Note that $a\vee b\neq 0$, so $\varphi(a\vee b)\neq \emptyset$. By assumption, there exists $\ultW\in\varphi(a\vee b)$ such that $\D{\ult}{\ultW}\subseteq \ultV$. But $c\in\D{\ult}{\ultW}$ which is a contradiction.    
\end{proof}

From lemmas \ref{lem:T-middle-non-empty}, \ref{lem:fundamental} and \ref{cond C3*} we get the dual spaces of pseudo-subordination algebras.

\begin{theorem}
    Let $\frA\defeq\klam{A,\condi}\in \CA$. $\frA$ is a pseudo-subordination algebra iff the relation $T_A$ of its dual space $\Es(\frA)$ satisfies\/\textup{:}
 if $T_A(\ult, Y, \ultV)$, then there exists $\ultW\in Y$ such that $T_A(\ult,\{\ultW\},\ultV)$.
\end{theorem}
\begin{proof}
($\Rarrow$) Assume $\frA$ is a pseudo-subordination algebra. If $T_A(\ult, Y, \ultV)$, then, by equation \eqref{C1-for-0}, for some proper filter $\fil$, $\D{\ult}{\fil}\subseteq\ultV$. Thus, there is an ultrafilter $\ultW$ such that $\fil\subseteq w$ and $\D{\ult}{\ultW}\subseteq\ultV$, i.e., $T_A(\ult,\{\ultW\},\ultV)$.

\smallskip

($\Larrow$) Suppose that $T_A$ satisfies the condition.  By Lemma \ref{lem:T-middle-non-empty}, we get that $\frA$ meets the equation \eqref{C1-for-0}, and by Lemma \ref{cond C3*}, we get that $A$ meets \eqref{C3-reversed}. Thus, $\frA$ is a pseudo-subordination algebra.
\end{proof}

The previous theorem guarantees that for a pseudo-subordination algebra $\frA \defeq \klam{A, \condi}$, the relation $T_A$ is actually induced by a ternary relation $S$ on $\Es(\frA)$ defined by
\[
S(x,y,z) \iffdef T_{A}(x, \{y\}, z).
\]
Conversely, $T_A$ can be recovered as 
\[
T_A(x,K,z) \text{ if and only if }K\in \closed(\tau_s) \text{ and there is } y \in K \text{ such that } S(x,y,z).
\]

\subsection{Correspondences} In this section we prove correspondences between the various axioms for $\condi$ presented in the previous subsection and first- and second-order properties of expanded Stone spaces of conditional algebras. To this end, we need the lemma below.

\begin{lemma}\label{lem:fundamental2} Let
$\frA\defeq\klam{A,\condi}\in\CA$. Let $\fil,\filH\in\Fi(A)$ and $\ultV\in\Ul(A)$.
If $\D{\filH}{\fil}\subseteq\ultV$, then there exists $\ult\in\varphi(\filH)$
such that $\D{\ult}{\fil}\subseteq\ultV$.
\end{lemma}
\begin{proof}
  Assume that $\D{\filH}{\fil}\subseteq\ultV$ and consider the family 
  \[
  \mathbf{F}\defeq\{\scrG\in\Fi(A):\D{\scrG}{\fil}\subseteq\ultV\text{ and }\filH\subseteq\scrG\}\,.
  \]
As $\filH\in\bfF$, and $\bfF$ is closed under unions of chains, by the Kuratowski-Zorn Lemma in $\bfF$ there exists a maximal filter $\filH^\ast$ extending $\filH$. Observe that $\filH^\ast$ is proper, as $\D{A}{\fil}=A\nsubseteq \ultV$. To show $\filH^\ast$ is prime, assume that $a\vee b\in \filH^\ast$ but $a\notin \filH^\ast$ and $b\notin \filH^\ast$.
Let us consider the filters $\filH^\ast_{a}$ and $\filH^\ast_{b}$ generated by $\filH^\ast\cup\{a\} $ 
and $\filH^\ast\cup\{b\}$, respectively. Neither of them is in $\bfF$ as they both extend  $\filH^\ast$ which is maximal in the set. In consequence $\D{\filH^\ast_{a}}{\fil}\nsubseteq \ultV$ and $\D{\filH^\ast_{b}}{\fil}\nsubseteq \ultV$, so
there are $x_1\in \fil$, $x_2\in \fil$ and $y_{1},y_{2}\notin \ultV$ such
that
\[
x_{1}\condi y_{1}\in \filH^\ast_a\quad\text{and}\quad x_{2}\condi y_{2}\in \filH^\ast_b.
\]
Take $y\defeq y_{1}\vee y_{2}\notin \ultV$ and $x\defeq x_1\wedge x_2\in \fil$.
By Lemma \ref{lem:orden}, $x\condi y\in \filH^\ast_a\cap \filH^\ast_b$. So, by the definitions of $\filH^\ast_a$ and $\filH^\ast_b$ there are $c_{1},c_{2}\in \filH^\ast$ such that $c_{1}\wedge a\leq x\condi y$
and $c_{2}\wedge b\leq x\condi y$. Let $c\defeq c_{1}\wedge c_{2}$.
Then $c\wedge (a\vee b)\in \filH^\ast$ and it is easy to see that
\[
c\wedge(a\vee b)\leq x \condi y\quad\text{and thus}\quad x\condi y\in \filH^\ast.
\]
Since $\D{\filH^\ast}{\fil}\subseteq \ultV$ and $x\in \fil$, we get $y\in \ultV$,
which is a contradiction. Thus, $\filH^\ast$ is prime, as required.
\end{proof}

Let us say that a property $\Phi$ of conditional algebras corresponds
to a first- or higher-order condition $\Psi$ expressible by a sentence in $\mathcal{L}_{T}$ if for every conditional
algebra $\frA\defeq\klam{A,\condi}$, $\frA$ has $\Phi$ (in symbols: $\frA\models\Phi$) if and
only if the expanded Stone space $\Es(\frA)=\langle\Ul(A),\topos,T_A\rangle$
has $\Psi$ (in symbols: $\Es(\frA)\models\Psi$). Of course, while speaking about a condition $\varphi$ for $\frA\in\CA$, we think about any equality involving $\condi$. All the conditions we presented so far are open, yet when we assume that $\frA$ satisfies, e.g., $a\condi a=1$ we mean that the closure of the formula holds in $\frA$. As all our constraints for $\condi$ are universal we will write---for example---$\frA\models a\condi a=1$ instead of $\frA\models (\forall a)\,a\condi a=1$ to save space and avoid notation cluttering. This will not lead to any ambiguity.

On the side of expanded spaces, we are interested in those properties of theirs that are expressible by means of $T_A$, and these, in some cases, are going to be quite complex. Again, for clarity reasons, we adopt the following conventions:
 \begin{enumerate}[label=(\arabic*),itemsep=0pt]
    \item the expression `$T\ult Y\ultV$' abbreviates `$T_A(\ult,Y,\ultW)$' (this, of course, applies to other first- or second-order letters),
    \item `$\forall \ult$' is to be interpreted as $\forall \ult\in \Ul(A)$, and `$\forall Y$', as $\forall Y\in\closed(\Ul(A))$ (also with the possibility to use different letters). 
\end{enumerate}

\begin{theorem}\label{th:correspondences-for-CAs} If $
\frA\defeq\klam{A,\condi}\in\CA$, then the following correspondences hold between formulas valid in $\frA$ and first- and second-order properties of $\Es(\frA)$\/\textup{:}
{\small
\begin{align}
\frA\models a\condi b\leq c\condi(a\condi b)\quad&\text{iff}\quad\Es(\frA)\models\forall \ult\ultV\ultW\forall Y\!Z(T\ult Y\ultV\,\&\,T\ultV Z\ultW\Rarrow T\ult Z\ultW)\,,\tag{A4}\label{eq:crrspndc-5}\\
\frA\models a\wedge(a\condi b)\leq b\quad&\text{iff}\quad\Es(\frA)\models\forall \ult T\ult\{\ult\}\ult\,,\tag{A5}\label{eq:crrspndc-6}\\
\frA\models a\condi b\leq\neg b\condi\neg a\quad&\text{iff}\quad\Es(\frA)\models\forall \ult\ultV\forall Y(T\ult Y\ultV  \Rarrow(\exists \ultW\in Y)\,T\ult\{\ultV\} \ultW)\,,\tag{A6}\label{eq:crrspndc-7}\\
\begin{split}
\frA\models\neg(a\condi b)\leq c\condi\neg(a\condi b)\quad&\text{iff}\quad 
\Es(\frA)\models\forall \ult\ultV\ultW\forall Y\!Z (T\ult Y\ultV\,\&\,T\ult Z\ultW\Rarrow T\ultV Z\ultW)\,,\label{eq:crrspndc-8}
\end{split}\tag{A7}\\[1em]
\begin{split}
\frA\models(1\condi (\neg a\lor b))\land (b\condi c)\leq a&{}\condi c\quad\text{iff}\\\Es(\frA)&\models\forall \ult\ultV\forall YZ(T\ult Y\ultV\,\&T(\ult,\Ul(A))\cap Y\subseteq Z \Rarrow T\ult Z\ultV)\,.\label{eq:crrspndc-11}
\end{split}\tag{A8}
\end{align}%
}
\end{theorem}
\begin{proof}
\eqref{eq:crrspndc-5} ($\Rarrow$) Consider $\ult,\ultV,\ultW\in \Ul(A)$ and $\fil,\filH\in\Fi(A)$
such that $T_A(\ult,\varphi(\fil),\ultV)$ and $T_A(\ultV,\varphi(\filH),\ultW)$.
Let $a,b\in A$ such that $a\condi b\in \ult$ and $a\in \filH$. Then
$1\condi(a\condi b)\in \ult$. As $T_A(\ult,\varphi(\fil),\ultV)$
and $1\in \fil$, $a\condi b\in \ultV$. Finally, from $T_A(\ultV,\varphi(\filH),\ultW)$
and $a\in \filH$ we get $b\in \ultW$.

\smallskip

 ($\Leftarrow$) Suppose that there exist $a,b,c\in A$ such that $a\condi b\nleq c\condi (a\condi b)$. Then, there exists an ultrafilter $\ult$ such that $a\condi b\in \ult$ but $c\condi (a\condi b)\notin \ult$. So, there exists $\ultV\in \Ul(A)$ such that $T_A(\ult,\varphi(c),\ultV)$ but $a\condi b\notin \ultV$. Thus, there exists $\ultW\in \Ul(A)$ such that $T_A(\ultV,\varphi(a),\ultW)$ but $b\notin \ultW$. By assumption, $T_A(\ult,\varphi(a),\ultW)$ but $b\in\D{\ult}{\upop a}$ and $b\notin \ultW$, a contradiction.

\smallskip

\eqref{eq:crrspndc-6} Immediate.

\smallskip

\eqref{eq:crrspndc-7} ($\Rarrow$) Let $\ult,\ultV\in \Ul(A)$ and $\fil\in \Fi(A)$ be such that $T_A(\ult,\varphi(\fil),\ultV)$. First note that if $b\in\D{\ult}{\ultV}$ then by assumption $\neg b\notin \fil$. Let us consider the set $I=\{b:\neg b\in\D{\ult}{\ultV}\}$. It is easy to see that $I$ is an ideal and $\fil\cap I=\emptyset$. So, there exists $\ultW\in \varphi(\fil)$ such that $\ultW\cap I=\emptyset$. We get that $\D{\ult}{\ultV}\subseteq \ultW$ and therefore $T_A(\ult,\{\ultV\},\ultW)$.

\smallskip

($\Leftarrow$) Suppose that there exist $a,b\in A$ be such that $a\condi b\nleq \neg b \condi \neg a$. Then, there exists $\ult\in \Ul(A)$ such that $a\condi b\in \ult$ but $\neg b\condi \neg a\notin \ult$. So, there exists $\ultV\in \Ul(A)$ such that $T_A(\ult,\varphi( \neg b),\ultV)$ but $a\in \ultV$. In addition, $b\in\D{\ult}{\ultV}$. By assumption, there exists $\ultW\in \varphi(\neg b)$ such that $T_A(\ult,\{\ultV\},\ultW)$, i.e., $\D{\ult}{\ultV}\subseteq \ultW$, which implies that $b\in \ultW$, a contradiction.

\smallskip

\eqref{eq:crrspndc-8} ($\Rarrow$) Let $\ult,\ultV,\ultW\in \Ul(A)$ and $\fil,\filH\in\Fi(A)$ be
such that $T_A(\ult,\varphi(\fil),\ultV)$ and $T_A(\ult,\varphi(\filH),\ultW)$.
Let $b\in A$ such that $b\in\D{\ultV}{\filH}$, i.e., there exists $a\in A$ such that $a\condi b\in \ultV$ and $a\in \filH$. We
prove that $a\condi b\in \ult$. To get a contradiction, suppose $\neg(a\condi b)\in \ult$. Thus $1\condi\neg(a\condi b)\in \ult$. As $T_A(\ult,\varphi(\fil),\ultV)$
and $1\in \fil$, we have $\neg(a\condi b)\in \ultV$, which is impossible.
Then $a\condi b\in \ult$, and from $T_A(\ult,\varphi(\filH),\ultW)$
and $a\in \filH$, we get $b\in \ultW$. Therefore $\D{\ultV}{\filH}\subseteq \ultW$.

\smallskip

($\Leftarrow$) Suppose that there exist $a,b,c\in A$ such that $\neg (a\condi b)\nleq c\condi \neg (a\condi b)$. Then, there exists $\ult\in \Ul(A)$ such that $\neg (a\condi b)\in \ult$ but $c\condi \neg (a\condi b)\notin \ult$. So, there exists $\ultV\in \Ul(A)$ such that $T_A(\ult,\varphi(c),\ultV)$  and $a\condi b\in \ultV$. From $a\condi b\notin \ult$, we get that there exists $\ultW\in \Ul(A)$ such that $T_A(\ult,\varphi(a),\ultW)$ but $b\notin \ultW$. By assumption, $T_A(\ultV,\varphi(a),\ultW)$, but since $a\condi b\in \ultV$, we obtain that $b\in \ultW$, a contradiction. 

\smallskip

\eqref{eq:crrspndc-11} ($\Rarrow$) Let $\ult,\ultV\in \Ul(A)$, $\fil_1,\fil_2\in \Fi(A)$ and suppose that $T_A(\ult, \varphi(\fil_1),\ultV)$ and $T_A(\ult,\varphi(1))\cap \varphi(\fil_1)\subseteq \varphi(\fil_2)$. We will prove that $\D{\ult}{\fil_2}\subseteq \ultV$. Let $c\in \D{\ult}{\fil_2}$. Then, there exists $b\in \fil_2$ such that $b\condi c\in \ult$. Consider the filter $\fil_3$ generated by $\D{\ult}{\{1\}}\cup \fil_1$. Note that $\fil_2\subseteq \fil_3$. Suppose, to get a contradiction, that there exists $a\in \fil_2$ such that $a\notin \fil_3$. So, there exists $\ultW\in \Ul(A)$ such that $\fil_3\subseteq \ultW$ and $a\notin \ultW$. Then, $\ultW\in T_A(\ult,\varphi(1))\cap \varphi(\fil_1)$ and by assumption, $\ultW\in \varphi(\fil_2)$ and we get $a\in\fil_2\subseteq \ultW$, a contradiction. So, $b\in \fil_3$ and it follows that there exist $d\in \D{\ult}{\{1\}}$ and $a\in \fil_1$ such that $d\land a\leq b$. It is immediate that $d\leq \neg a \lor b$ and by \ref{eq:cond-order-preserving}, $1\condi d\leq 1\condi (\neg a \lor b)$. Since $d\in \D{\ult}{\{1\}}$, $1\condi d\in \ult$ and thus $1\condi (\neg a\lor b)\in \ult$. We get $1\condi (\neg a\lor b)\land (b\condi c)\in \ult$ and by assumption, $(1\condi (\neg a\lor b))\land (b\condi c)\leq a\condi c$. It follows that $c\in \D{\ult}{\fil_1}$ and since  $T_A(\ult, \varphi(\fil_1),\ultV)$, $c\in \ultV$. 

\smallskip

($\Leftarrow$) Suppose that there exist $a,b,c\in A$ such that $(1\condi (\neg a\lor b))\land (b\condi c)\nleq a\condi c$. Then, there exists $\ult\in \Ul(A)$ such that $(1\condi (\neg a\lor b))\land (b\condi c)\in \ult$ but $a\condi c\notin \ult$. So, there exists $\ultV\in \Ul(A)$ such that $T_A(\ult,\varphi(a),\ultV)$ and $c\notin \ultV$. We will see that $T_A(\ult,\varphi(1))\cap \varphi(a)\subseteq \varphi(b)$. Let $\ultW\in T_A(\ult,\varphi(1))\cap \varphi(a)$. Then, $\D{\ult}{\{1\}}\subseteq \ultW$ and $a\in w$. By assumption, $ \neg a\lor b\in \D{\ult}{\{1\}}\subseteq \ultW$, and we get that $\neg a\lor b\in \ultW$. Thus, since $a\in \ultW$, it follows that $b\in \ultW$ and we get $\ultW\in\varphi(b)$. By assumption, $T_A(\ult,\varphi(b),\ultV)$ and it follows $c\in \D{\ult}{\upop b}\subseteq \ultV$, a contradiction.
\end{proof}

\subsection{Canonicity}

Theorem \ref{th:correspondences-for-CAs} together with the following Theorem
\ref{th:conditions-to-prove-canonicity} guarantee
that all the varieties considered in this work are canonical in
the sense we explain below. 

\begin{definition}
An equation
is \textit{canonical} (i.e., \emph{$\pi$-canonical}) for conditional algebras if whenever
it is valid in a conditional  algebra $\frA\defeq\klam{A,\condi}$, then it is valid
in the full complex algebra $\Em(\frA))=\langle\mathcal{P}(\Ul(A)),\condi_{T_A}\rangle$
of its conditional space $\Es(\frA)\defeq\langle\Ul(A),\topos,T_A\rangle$. 
\end{definition}
 
\begin{lemma}\label{lem:T-via-condi}
    Let $\frX\defeq\langle X,\tau,T\rangle$ be a conditional space. Let $x\in X$, $Y\in \closed(\tau)$ and $W\subseteq X$. Then, $T(x,Y)\subseteq W$ if and only if $x\in Y\Tcondi W$.
\end{lemma}
\begin{proof}
    ($\Rarrow$) Suppose that $T(x,Y)\subseteq W$. Let $T(x,Z,y)$ such that $Z\subseteq Y$. Let $U\in \clopen(\tau)$ such that $Y\subseteq U$. Since $Z\subseteq U$, we get that $T(x,U,y)$. Thus, by condition \ref{T3}, $T(x,Y,y)$ and it follows $y\in W$.
    
    ($\Larrow$) $x\in Y\Tcondi W$. Since $Y\subseteq Y$, $T(x,Y)\subseteq W$.
\end{proof}

\begin{theorem}\label{th:conditions-to-prove-canonicity}
The following equivalences hold between any conditional space $\frX\defeq\langle X,\tau,T\rangle$ and its complex algebra $\Cm(\frX)=\klam{\power(\Ul(X)),\Tcondi}$ \textup{(}for brevity, we omit the subscript $T$ at $\condi$\textup{)}\/\textup{:} 
\allowdisplaybreaks
\begin{align}
\begin{split}
\Cm(\frX) \models(a\condi c)\wedge(b\condi c)&{}\leq(a\vee b)\condi c
\quad\text{iff}\\ 
&\frX \models\forall xy\forall Y\neq \emptyset(TxYy\Rarrow(\exists z\in Y)\,Tx\{z\}y) \label{eq:crrspndc-II-3}
\end{split}\tag{T3$^\ast$}
\\[1em]
 \Cm(\frX) \models a\condi b\leq c\condi(a\condi b)
&{}\quad\text{iff}\quad \frX \models\forall xyz\forall YZ(TxYy\,\&\,TyZz\Rarrow TxZz)\tag{T4}\label{eq:crrspndc-II-5}\\
 \Cm(\frX) \models a\wedge(a\condi b)\leq b&{}\quad\text{iff}\quad \frX \models\forall xTx\{x\}x\tag{T5}\label{eq:crrspndc-II-6}\\
 \Cm(\frX) \models a\condi b\leq\neg b\condi\neg a
&{}\quad\text{iff}\quad\frX \models\forall xy\forall Y(TxYy\Rarrow(\exists z\in Y)\,Tx\{ y\} z)\tag{T6}\label{eq:crrspndc-II-7}\\[1em]
\begin{split}
 \Cm(\frX) \models\neg(a\condi b)\leq c\condi\neg&{}(a\condi b)
\quad\text{iff}\quad\\
&\frX \models\forall xyz\forall Y\!Z(TxYy\,\&\,TxZz\Rarrow TyZz)\,.\label{eq:crrspndc-II-8}
\end{split}\tag{T7}
\\[1em]
\begin{split}
 \Cm(\frX) \models (1\condi(\neg a \vee b))\wedge (b&{}\condi c){}\leq a\condi c
\quad\text{iff}\quad\\
&\frX \models \forall xy\forall YZ(Tx Y y\,\&T(x,X)\cap Y\subseteq Z \Rarrow Tx Zy)\,.\label{eq:crrspndc-II-11}
\end{split}\tag{T8}
\end{align}
\end{theorem} 
\begin{proof}
\eqref{eq:crrspndc-II-3} ($\Rarrow$) Suppose that  $\Cm(\frX) \models(a\condi c)\wedge(b\condi c)\leq(a\vee b)\condi c$. Let $T(x,Y,y)$ and assume that $Y\neq\emptyset$. Assume, to get a contradiction, that for all $z\in Y$, $T^c(x,\{z\},y)$. Fix a $z'\in Y$. By condition \ref{T3}, there exists $V_{z'}\in\clopen(\tau)$ such that $z'\in V_{z'}$ and $T^c(x,V_{z'},y)$. Let us consider the family of all $V_z\in \clopen(\tau)$ such that $z\in Y$. Then, 
\[
Y\subseteq \bigcup_{z\in Y}V_z\,,\qquad\text{and in consequence}\quad \displaystyle Y\subseteq \bigcup_{i=1}^n V_{z_i}
\]
for a finite family of elements of $Y$, since $Y$ is compact (as a closed subset of a compact space). From $y\notin T(x,V_{z_i})$ it follows that $x\in V_{z_i}\Tcondi \{y\}^c$ for all $a\leq i\leq n$. By assumption, 
\[
\displaystyle x\in \bigcap_{i=1}^n ( V_{z_i}\Tcondi \{y\}^c) \subseteq \left(\bigcup_{i=1}^n V_{z_i}\right) \Tcondi \{y\}^c
\]
and since $Y\subseteq \bigcup_{i=1}^n V_{z_i}$, we get that $y\in \{y\}^c$, a contradiction.

\smallskip

($\Larrow$) Let $W_1,W_2,W_3\subseteq X$ and $x\in (W_1\Tcondi W_3)\cap (W_2\condi W_3)$. Let $T(x,Y,y)$ be such that $Y\subseteq W_1\cup W_2$. If $Y=\emptyset$, we get that $y\in W_3$. So assume that $Y\neq \emptyset$, which entails existence of $z\in Y$ such that $T(x,\{z\},y)$. Since $z\in W_1$ or $z\in W_2$, it follows that $y\in W_3$.

\smallskip

\eqref{eq:crrspndc-II-5} ($\Rarrow$) Suppose that $\Cm(\frX) \models a\condi b\leq c\condi(a\condi b)$. Let $x,y,z\in X$ and $Y,Z\in \closed(\tau)$ be such that  $T(x,Y,y)$ and $T(y,Z,z)$. Assume that $T^c(x,Z,z)$. From, $T(x,Z)\subseteq \{z\}^c$, we get that $x\in Z\Tcondi \{z\}^c$. By assumption, $x\in Y \Tcondi (Z\Tcondi \{z\}^c)$ and it follows that $y\in Z\Tcondi \{z\}^c$. Thus, $z\in \{z\}^c$ which is a contradiction.

\smallskip

($\Larrow$) Let $W_1,W_2,W_3\subseteq X$ and let $x\in (W_1\Tcondi W_2)$. Suppose $T(x,Y,y)$ and $Y$ is such that $Y\subseteq W_3$. We will prove that $y\in (W_1\Tcondi W_2)$. Let $T(y,Z,z)$ such that $Z\subseteq W_1$. By assumption, $T(x,Z,z)$, so $z\in W_2$. It follows that $y\in W_1\Tcondi W_2$, and thus $x\in W_3\Tcondi (W_1\Tcondi W_2)$.

\smallskip

\eqref{eq:crrspndc-II-6} ($\Rarrow$) Suppose that $\Cm(\frX) \models a\wedge(a\condi b)\leq b$ and let $x\in X$. Suppose that $T^c(x,\{x\},x)$. Then $x\in \{x\}\Tcondi \{x\}^c$. By assumption, $x\in \{x\}\cap (\{x\}\Tcondi \{x\}^c)\subseteq \{x\}^c$, a contradiction. 

\smallskip

($\Larrow$) Let $W_1,W_2\subseteq X$ and $x\in W_1\cap (W_1\Tcondi W_2)$. So, $\{x\}\subseteq W_1$ and by assumption $T(x,\{x\},x)$. Therefore, $x\in W_2$.

\smallskip

\eqref{eq:crrspndc-II-7} ($\Rarrow$) Suppose that $\Cm(\frX) \models a\condi b\leq\neg b\condi\neg a$ and let $T(x,Y,y)$. Suppose that for all $z\in Y$, $T^c(x,\{y\},z)$. So, $T(x,\{y\})\subseteq Y^c$ and we get that $x\in \{y\}\Tcondi Y^c$. By assumption, $x\in Y\Tcondi \{y\}^c$ which entails $y\in\{y\}^c$, a contradiction.

\smallskip

($\Larrow$) Let $W_1,W_2\subseteq X$ and let $x\in X$. Suppose $x \in W_1\Tcondi W_2$ and $T(x,Y,y)$, where $Y$ is such that $Y\subseteq W_2^c$. By assumption, there exists $z \in Y$ with  $T(x, \{y\}, z)$. If $y\in W_1$, we get that $z\in W_2$, a contradiction. Therefore $y\in W_1^c$.

\smallskip

\eqref{eq:crrspndc-II-8} ($\Rarrow$) Suppose that $\Cm(\frX) \models\neg(a\condi b)\leq c\condi\neg(a\condi b)$. Let  $T(x,Y,y)$ and $T(x,Z,z)$. Suppose $T^c(y,Z,z)$. Then, $T(y,Z)\subseteq\{z\}^c$ and we get $y\in Z\Tcondi \{z\}^c$. On the other hand, we have $x\notin Z\Tcondi \{z\}^c$. By assumption, $x\in Y\Tcondi (Z\Tcondi \{z\}^c)^c$ and it follows that $y\notin Z\Tcondi \{z\}^c$, a contradiction.

\smallskip

($\Larrow$) Let $W_1,W_2\subseteq X$, $x \notin W_1\Tcondi W_2$ and $T(x,Y,y)$ where $Y\subseteq W_3$. So, there exist $Z\in \closed(\tau)$ and $z\in X$ such that $Z\subseteq W_1$ but $z\notin W_2$. By assumption, $T(y,Z,z)$ and it follows that $y\notin W_1\Tcondi W_2$. 

\smallskip

\eqref{eq:crrspndc-II-11} ($\Rarrow$) Suppose that $\Cm(\frX) \models (1\condi(\neg a \vee b))\wedge (b\condi c)\leq a\condi c$ and let $T(x,Y,y)$ where $Y$ is such that $T(x,X)\cap Y \subseteq Z$. Then, $T(x,X)\subseteq Y^c\cup Z$. It follows that $x\in X \condi_T Y^c\cup Z$. In addition, we get that $x\in Z \condi_T T(x,Z)$. By assumption, it follows that $x\in Y\condi_T T(x,Z)$, i.e., $T(x,Y)\subseteq T(x,Z)$. Since $y\in T(x,Y)$, we get that $T(x,Z,y)$.

\smallskip

($\Larrow$) Now, let $W_1,W_2,W_3\subseteq X$ and suppose that $x\in (X\condi_T (W_1^c\cup W_2))\cap (W_2\condi_T W_3)$. Let $Y\in \closed(\tau)$ be such that $Y\subseteq W_1$. We will show that $T(x,Y)\subseteq W_3$. To this end, let $y\in T(x,Y)$. Since $x\in X\condi_T (W_1^c\cup W_2)$, we get that $T(x,X)\subseteq W_1^c\cup W_2$, i.e., $T(x,X)\cap W_1\subseteq W_2$. Then, $T(x,X)\cap Y\in\closed(\tau)$ and $T(x,X)\cap Y\subseteq T(x,X)\cap W_1\subseteq W_2$ and $T(x,X)\cap Y\subseteq T(x,X)\cap Y$. By assumption,  $T(x,T(x,X)\cap Y,y)$. From $x\in W_2\condi_T W_3$ and $T(x,X)\cap Y\subseteq W_2$, we get that $y\in W_3$.
\end{proof}

\begin{corollary}
    The varieties $\PSB$, $\PsC$, $\SIA$ and $\StwoIA$ are closed under canonical extensions. 
\end{corollary}

\section{Summary and further work}

We have shown that the techniques developed for monotonic operators in \citep{Celani-TDFBAWANNMO,Menchon-PhD,Celani-et-al-MDS} turned out to be prolific enough to develop the dualities and the canonical extensions for the variety $\CA$ whose elements are algebraic models of a system of basic conditional logic. Moreover, we have demonstrated that the same techniques are applicable to well-known subvarieties of $\CA$.

In a future installment of this work, we aim to develop the theory of quasi-conditional algebras in the spirit of quasi-modal operators of Celani's \citeyearpar{Celani-QMA}, and we want to investigate the relation of both conditional and quasi-conditional algebras to weak extended contact algebras from \citep{Balbiani-et-al-RRTFECA}.

\section*{Acknolwedgements}

This research was funded by (a) the National Science Center (Poland), grant number~2020/39/B/HS1/00216 and (b) the MOSAIC project (EU H2020-MSCA-RISE-2020 Project 101007627).

We would like to thank two anonymous referees for the excellent reports, which were insightful, substantial and helped us improve the paper a lot.

\nocite{Weiss-BICL,Ciardelli-et-al-ICL}

\bibliographystyle{apalike}

\providecommand{\noop}[1]{}

\end{document}